\documentclass[oneside, a4paper,11pt,reqno]{amsart}

\usepackage[T1]{fontenc}

\usepackage{verbatim}
\usepackage{amsmath}
\usepackage{amsfonts}
\usepackage{color}
\usepackage{amssymb}
\newtheorem{theo}{Theorem}
\newtheorem{prop}[theo]{Proposition}

\newtheorem{lem}[theo]{Lemma}

\newcommand {\pare}[1] {\left( {#1} \right)}
\newcommand {\cro}[1] {\left[ {#1} \right]}
\newcommand {\acc}[1] {\left\{ {#1} \right\}}
\newcommand {\nor}[1] { \left\| {#1} \right\|}
\newcommand {\nors}[1] {\nor{#1}^{*}}
\newcommand {\bra}[1] { \left\langle {#1} \right\rangle}
\newcommand {\bras}[1] { \bra {#1}^{*}}
\newcommand {\abs}[1] {\left\lvert {#1} \right\rvert}
\newcommand {\va}[1] {\left| {#1} \right|}
\def \E {\mathbb{E}}
\def \N  {\mathbb{N}} 
\def \P {\mathbb{P}}
\def \R  {\mathbb{R}} 
\def \Z  {\mathbb{Z}}

\def \AA {{\mathcal A}}
\def \EE {{\mathcal E}}
\def \FF {{\mathcal F}}
\def \LL {{\mathcal L}}

\def \XX {{\mathcal X}}
\def \MM {{\mathcal M}}
\def \UU {{\mathcal U}}
\def \VV {{\mathcal V}}
\def \SS {{\mathcal S}} 

\newcommand {\som}[2] { \displaystyle{\sum_{#1}^{#2}}}
\newcommand {\pro}[2] {\displaystyle{\prod_{#1}^{#2}}}

\def \la {\lambda}
\def \ind {\hbox{\tt 1\hskip-.33em l}}
\def \bx {\bar{x}}
\def \bX {\bar{\XX}}
\def \bP {\bar{P}}
\def \by {\bar{y}}
\def \bz {\bar{z}}
\def \bf {\bar{f}}
\def \bR {\bar{R}}

\def \bLL {\bar{\LL}}

\def \brX {\breve{\XX}}
\def \brx {\breve{x}}
\def \bry {\breve{y}}
\def \brz {\breve{z}}
\def \brf {\breve{f}}
\def \brR {\breve{R}}

\def \Id {\mbox{Id}}
\def \Tr {\mbox{Trace}}
\def \Vol {\mbox{Vol}}
\def \Span {\mbox{Span}}

\begin{document}

\title[Random forests and intertwining]{Approximate and exact solutions\\ of intertwining equations\\ through random spanning forests}

\author{Luca Avena}
\thanks{L. A. was partially supported by NWO Gravitation Grant 024.002.003-NETWORKS}
\address{Universiteit Leiden, Niels Bohrweg 1, 2333 CA Leiden, Netherlands}
\email{l.avena@math.leidenuniv.nl}

\author{Fabienne Castell} 
\address{
 Aix-Marseille Universit\'e, CNRS, Centrale Marseille. I2M UMR CNRS 7373. 39, rue Joliot Curie. 13 453 Marseille Cedex
13. France.}
\email{fabienne.castell@univ-amu.fr}

\author{Alexandre Gaudilli\`ere} 
\address{
 Aix-Marseille Universit\'e, CNRS, Centrale Marseille. I2M UMR CNRS 7373. 39, rue Joliot Curie. 13 453 Marseille Cedex
13. France.}
\email{alexandre.gaudilliere@math.cnrs.fr} 

\author{Clothilde M\'elot} 
\address{
 Aix-Marseille Universit\'e, CNRS, Centrale Marseille. I2M UMR CNRS 7373. 39, rue Joliot Curie. 13 453 Marseille Cedex
13. France.}
\email{clothilde.melot@univ-amu.fr}

\date{\today}


\begin{abstract}
For different reversible Markov kernels on finite state spaces,
we look for families of probability measures
for which the time evolution almost remains in their convex hull.
Motivated by signal processing problems
and metastability studies
we are interested in the case when the size of such families
is {\em smaller} than the size of the state space,
and we want such distributions to be with small overlap among them.
To this aim we introduce a {\em squeezing} function
to measure the common overlap of such families,
and we use random forests to build random approximate solutions
of the associated intertwining equations
for which we can bound from above the expected values
of both squeezing and total variation errors.
We also explain how to modify some of these approximate solutions 
into exact solutions by using those eigenvalues
of the associated Laplacian
with the largest absolute values.
\end{abstract}

\maketitle

\section{Main results, motivations and heuristic}

The aim of this work is to build exact and approximate solutions 
of certain intertwining equations between Markov kernels on finite state spaces.
The intertwining equations we look at are related to the two following problems.
First, we want to build wavelet-like multiresolution schemes
for signal processing on arbitrary weighted graphs.
Second, we want to make sense of the notion of metastability without asymptotic,
in a finite setup where no large-volume or low-temperature limits are in place.
We will partially address these problems by giving 
``good approximate solutions'' of the intertwining equations 
we are about to describe together with the structure of the paper.
We anticipate that, as far as the first problem is concerned,
the results we derive here form the main guideline
for the filtering and subsampling operations 
of a full multiresolution scheme
we derive in a forthcoming paper.

\subsection{Intertwining equations}
Given an irreducible stochastic matrix $P$,
which is associated with the generator ${\mathcal L}$
of a continuous time process $X$ on a finite state space ${\mathcal X}$
(see Section~\ref{genova} for precise definitions),
we look at solutions $(\Lambda, \bar P)$  
of the intertwining equations
\begin{equation}
\label{DF.eq}
{\Lambda} P = \bar{P} {\Lambda} ,
\end{equation}
and, for $q' > 0$,
\begin{equation}
\label{meta.eq}
{\Lambda} K_{q'} = \bar{P} {\Lambda} ,
\end{equation}
where 
\begin{itemize} 
\item $\bar{P}$ is a stochastic matrix defined on some finite state space $\bar{\XX}$;
\item ${\Lambda}: \bar{\XX} \times \XX \rightarrow [0, 1]$ is a rectangular stochastic matrix;
\end{itemize}
and 
$K_{q'}$  is the transition kernel on $\XX$ given by 
\begin{equation}
\label{Kq.def}
 K_{q'}(x,y) :=\P_x(X(T_{q'}) = y ) = q' (q' \, \Id - \LL)^{-1}(x,y) \, , 
\end{equation}
with $T_{q'}$ an exponential random variable with parameter $q'$ that is independent of~$X$.  

Solving equation \eqref{DF.eq} amounts to find a family of probability measures 
$\nu_{\bar x} = \Lambda(\bar x, \cdot)$  
on ${\mathcal X}$ and such that, for some stochastic matrix $\bar P$,
\begin{equation} \label{banzai}
	\nu_{\bar x}P = \Lambda P(\bar x, \cdot) = \bar P \Lambda(\bar x, \cdot) = \sum_{\bar y \in \bar {\mathcal X}} \bar P(\bar x, \bar y) \nu_{\bar y}
	.
\end{equation}
In other words the one step evolution of the $\nu_{\bar x}$'s
have to remain in their convex hull.
Solving Equation~\eqref{meta.eq} is the same, except that the ``one step evolution''
has now to be considered in continuous time and on time scale $1 / q'$.
In both cases a trivial solution is always given by taking all the $\nu_{\bar x}$
equal to the equilibrium measure $\mu$.

Intertwining relations, restricted to measures $\nu_{\bar x}$ with disjoint support,
appeared in the context of diffusion processes in the paper by Rogers and Pitman \cite{RP}, as 
a tool to state identities in laws. This method was later successfully applied 
to many other examples (see for instance \cite{CPY},
\cite{MY}). In the context of Markov chains, intertwining was used by Diaconis and Fill \cite{DF}
without the disjoint support restriction
to build strong stationary times and to control convergence rates to equilibrium. 
At the time being, applications of intertwining include random matrices \cite{DMDY}, particle systems \cite{War}...

Such intertwining relations have often been considered from \cite{DF} with an absorbing point
for~$\bar P$ in~$\bar{\mathcal X}$ and with size $m$ of $\bar{\mathcal X}$ being (much) larger than
or equal to the size $n$ of ${\mathcal X}$.
Motivated by signal processing and metastability problems
(see Section~\ref{violet}), in this paper 
we are instead interested in the case where
\begin{itemize}
	\item the size $m$ of $\bar{\mathcal X}$ is smaller than the size $n$ of ${\mathcal X}$,
	\item $\bar P$ is irreducible,
	\item the probability measures $\bigl(\nu_{\bar x} : \bar x \in \bar{\mathcal X}\bigr)$
		are linearly independent and have small ``joint overlap''.
\end{itemize}
We will define the {\em squeezing} of a collection of probability measures 
to control this overlap (see Section~\ref{inverno}) and a small ``joint overlap''
will correspond to little squeezed probability measures.
We will further assume that $P$ and ${\mathcal L}$ are reversible 
with respect to $\mu$ and we will see in Section~\ref{impianto}
that, for any reversible stochastic kernel $P$ 
with non-negative eigenvalues
and for any positive $m < n$,
non-degenerate solutions of Equation~\eqref{DF.eq} 
with $|\bar{\mathcal X}| = m$
always exist.
By ``non-degenerate solutions'' we mean linearly independent probability measures 
such that Equation~\eqref{banzai} holds for some irreducible~$\bar P$.
But we will argue that exact solutions tend to be squeezed solutions.
Then, rather than looking at the less squeezed solutions
in the large space of all solutions for a given $m$,
we will first consider approximate solutions with small squeezing.
To this aim we will make use of random spanning forests
to build random approximate solutions
for which we will be able to bound 
both the expected value of an error term in intertwining Equation~\eqref{DF.eq}
and  the expected value of the squeezing (Theorem~\ref{cornet}).
Then we will use the same random forests
to build random approximate solutions of Equation~\eqref{meta.eq}
with no overlap, i.e., with disjoint support (Theorem~\ref{TVmeta.theo}). 
Assuming knowledge of the $n - m$ largest eigenvalues of $-{\mathcal L}$,
we will finally see how to modify such an approximate solution of~\eqref{meta.eq}
with $m$ probability measures $\nu_{\bar x}$
into exact solutions for $q'$ small enough (Theorem~\ref{cup}).

\subsubsection*{Structure of the paper}
We fix some notation in Section~\ref{genova}
before defining the squeezing of a probability measure family in Section~\ref{inverno}.
We introduce random forests in Section~\ref{notation.sub}
to give our main results in Section~\ref{grey}.
We detail our motivations, linking signal processing and metastability studies,
in Section~\ref{violet}
and we give some heuristics in Section~\ref{pink}.
In Section~\ref{campi} we prove some preliminary results,
and we give the proofs of our three main theorems
in the three last sections.
We conclude with an appendix 
that contains the proof of the main statement
that links metastability studies with Equation~\eqref{banzai}.

\subsection{Functions, measures, Markov kernel and generator}\label{genova}
Let  $\XX$ be  a finite space with cardinality $\abs{\XX}= n$. We consider  an irreducible 
  continuous time Markov process
$(X(t), t \ge 0)$ on  $\XX$,  with generator 
 $\LL$: 
 \begin{equation}
 \label{Generateur.def}  \LL f(x) := \sum_{y \in \XX} w(x,y) (f(y)- f(x)), 
 \end{equation}
where $f: \XX \rightarrow \R$ is an arbitrary function, and $w: \XX \times \XX \rightarrow [0,+ \infty[$ gives the 
transition rates. 
For $x \in \XX$, let  
\[ w(x) := \sum_{y \in \XX \setminus \acc{x}} w(x,y) \, . 
\]
Note that $\LL$ acts on functions as the matrix, still denoted by $\LL$:
\[ \LL(x,y)=w(x,y)  \mbox { for } x \neq y \, ; \, \, \LL(x,x)=-w(x) \, . 
\]
Let $\alpha > 0$ be defined by 
\begin{equation}
\label{MaxTaux.eq}
 \alpha = \max_{x \in \XX} w(x) \, . 
\end{equation}
Hence, $P:= \LL/\alpha+\Id$ is an irreducible stochastic matrix, 
and we denote by $(\hat{X}_k, k \in \N)$ a discrete time Markov
chain with transition matrix $P$. The process $(X(t), t \ge 0)$ can be constructed from $(\hat{X}_k, k \in \N)$ and an 
independent Poisson process $(\tau_i, i > 0)$ on $\R^+$ with rate $\alpha$.
At each event of the Poisson process,  $X$ moves according to the trajectory of $\hat{X}$,
i.e., with $\tau_0 = 0$:
\[ X(t) = \sum_{i=0}^{+\infty} \hat{X}_i \ind_{\tau_i \leq t < \tau_{i+1}} \, . 
\]
We assume that $X$ is reversible with respect to the probability measure $\mu$ on $\XX$, i.e.
\begin{equation}
\label{equilibre.eq}
 \forall x, y  \in \XX, \, \, \mu(x) w(x,y) = \mu(y) w(y,x) \, . 
\end{equation} 
The process $X$ being irreducible, $\mu$ is strictly positive.
The operator $- \LL$ is self-adjoint and positive; we denote by $(\lambda_i; i=0, \cdots, n-1)$
the real eigenvalues of $-\LL$ in increasing order. It follows from  the fact that $P$ is  irreducible
 that 
\begin{equation}
\label{vp.eq}
 0 = \la_0 < \la_1 \leq \la_2 \cdots \leq \la_{n-1} \leq  2 \alpha \, .
\end{equation}

A function $f$ on $\XX$ will be seen as a column vector, whereas a signed measure on $\XX$ will be seen as a 
row vector. For $p \geq 1$,  $\ell_p(\mu)$ is the space of functions endowed with the norm 
\[ \nor{f}_p = \pare{\sum_{x \in \XX} \va{f(x)}^p  \mu(x)}^{1/p}\, .
\]
The scalar product of two functions $f$ and $g$ in $\ell_2(\mu)$ is 
\[ \bra{f,g} = \sum_{x \in \XX} f(x) g(x) \mu(x) \,  
\] 
The corresponding norm is  denoted by $\nor{\cdot}$. When $f$ is a function and $\nu$ is a signed measure, the duality bracket between $\nu$ and $f$ is 
\[ \bra{\nu | f} = \sum_{x \in \XX}  \nu(x) f(x) \, . 
\]
$\ell_p^*(\mu)$ denotes the dual space of $\ell_p(\mu)$ with respect to $\bra{\cdot | \cdot}$. It is the space 
of signed measures endowed with the norm:
\[ \nor{\nu}_p^* = \pare{\sum_{x \in \XX} \va{\frac{\nu(x)}{\mu(x)}}^{p^*} \mu(x) }^{1/p^*} \, 
\]
where $p^*$ is the conjugate exponent of $p$: $1/p + 1/p^*=1$.
$\ell_p^*(\mu)$ is identified with $\ell_{p^*}(\mu)$ through the isometry: $\nu \in \ell_p^*(\mu) \mapsto 
\nu^* \in \ell_{p^*}(\mu)$, where $\nu^*(x)= \nu(x)/\mu(x)$ is the density of $\nu$ with
respect to $\mu$. The inverse of this isometry is still denoted by $^*$. 
It associates to a function $f \in \ell_p(\mu)$, the signed measure $f^* \in \ell_{p^*}^*(\mu)$ whose density with respect to $\mu$ is $f$: 
$f^*(A)= \sum_{x \in A} \mu(x) f(x)$ for all subset $A$ of $\XX$. $\ell_2^*(\mu)$
is an Euclidean space whose scalar product is denoted by:
\[
\bras{\nu,\rho}:=\sum_{x \in \XX} \nu(x) \rho(x) \frac{1}{\mu(x)} = \bra{\nu^*,\rho^*} \, .
\]
The corresponding norm is  denoted by $\nors{\cdot}$. For $\nu \in \ell_2^*(\mu)$ and 
$f \in \ell_2(\mu)$, one gets 
\[ \bra{\nu | f} =\bra{\nu,f^*}^* = \bra{\nu^*,f} \, .
\]

\subsection{Squeezing of a collection of probability measures}\label{inverno} 
For some finite space $\bar{\mathcal X}$ of size $m < n$, let $(\nu_{\bx} : \bx \in \bX)$ be a collection of $m$ probability measures on $\XX$
which is identified with the matrix~${\Lambda}$, the row vectors of which are the $\nu_{\bar x}$'s:
$\Lambda(\bar x, \cdot) = \nu_{\bar x}$ for each $\bar x$ in $\bar{\mathcal X}$.
Since these measures form acute angles between them ($\langle \nu_{\bar x}, \nu_{\bar y}\rangle^* \geq 0$ for 
all $\bar x$ and $\bar y$ in $\bar{\mathcal X}$) and have disjoint supports if and only if they are orthogonal,
one could use the volume of the parallelepiped they form to measure their ``joint overlap''.
The square of this volume is given by the determinant of the Gram matrix:
$$
	{\rm Vol}(\Lambda) = \sqrt{\det\Gamma},
$$
with $\Gamma$ the square matrix on $\bX$ with entries
$\Gamma(\bx, \by) = \bras{\nu_{\bx},\nu_{\by}}$, that is
\begin{equation}
\label{Gamma.def}
\Gamma:= {\Lambda} D(1/\mu) {\Lambda}^t \, ,
\end{equation}
where $D(1/\mu)$ is the diagonal matrix with entries given by $(1/\mu(x), x \in \XX)$, and 
${\Lambda}^t$ is the transpose of ${\Lambda}$. 
Loosely speaking, the less overlap, the largest the volume.

We will instead use the {\em squeezing} of $\Lambda$, that we defined by
\begin{equation}
\label{Flat.eq}
\SS({\Lambda}):=\left\{\begin{array}{ccc}
&+\infty&\mbox{ if }\det(\Gamma)=0,\\
&\sqrt{\Tr\big(\Gamma^{-1}\big)} \in\,]0,+\infty[&\mbox{ otherwise,}
\end{array}\right. 
\end{equation} 
to measure this ``joint overlap''.
We call it ``squeezing'' because the $\nu_{\bar x}$ and the parallelepiped they form
are squeezed when ${\mathcal S}(\Lambda)$ is large.
This is also 
the half diameter of the rectangular parallelepiped
that circumscribes the ellipsoid defined by the Gram matrix $\Gamma$ :
this ellipsoid is squeezed too when ${\mathcal S}(\Lambda)$ is large.
We note finally that our squeezing controls the volume of $\Lambda$.
Indeed, by comparison between harmonic
and geometric mean applied to the eigenvalues of the Gram matrix, small squeezing implies large volume:
${\rm Vol}(\Lambda)^{1 / n} {\mathcal S}(\Lambda) \geq \sqrt n$.
We will also show in Section~\ref{campi}:
\begin{prop} 
\label{flatness.prop}
 Let $(\nu_{\bx}, \bx \in \bX)$ be a collection of $m$ probability measures on $\XX$. 
\begin{enumerate}
\item We have
\begin{equation} 
\label{Trace.ineq}
\SS({\Lambda}) \geq \sqrt{ \sum_{\bx \in \bX} \frac{1}{\nor{\nu_{\bx}}^{*2}}} \, . 
\end{equation} 
Equality holds if and only if  the $(\nu_{\bx}, \bx \in \bX)$ are orthogonal.
\item Assume that $\mu$ is a convex combination of the $(\nu_{\bx}, \bx \in \bX)$. 
Then,  
\[ \SS({\Lambda}) \geq 1 \, . 
\]
Equality holds  if and only if the $(\nu_{\bx}, \bx \in \bX)$ are orthogonal.
\end{enumerate}   
\end{prop}
\noindent
{\it Comment:} $\SS({\Lambda})$ is thus maximal when the $(\nu_{\bx}, \bx \in \bX)$ are linearly dependent, and
minimal when they are orthogonal. Moreover, we know the minimal value of $\SS({\Lambda})$, 
when $\mu$ is a convex combination of the $(\nu_{\bx}, \bx \in \bX)$.  Note that this is necessarily the case
if the convex hull of the $\nu_{\bar x}$ is stable under~$P$,
i.e., when $\Lambda P = \bar P \Lambda$ for some stochastic $\bar P$.
Indeed it is then stable under $e^{t{\mathcal L}}$ for any $t > 0$
and the rows of $\Lambda e^{t{\mathcal L}}$ converge to $\mu$ when $t$ goes to infinity.
Note also that we are using ``$\ell_2(\mu)$ computations'' (through the Gram matrix)
to define the squeezing of measures that are normalized in $\ell_1(\mu) \sim \ell_\infty^*(\mu)$
(these are {\em probability} measures). This proposition shows that such a mixture of norms is not meaningless.

\subsection{Random forests} \label{notation.sub}
Note that the weight function $w$ induces a structure of oriented graph on~$\XX$, $e=(x,y)$ being
an oriented edge if and only if $w(e):=w(x,y)>0$.  Let $\EE$ be the set of oriented edges, and $G=(\XX,\EE)$
the oriented graph just defined. An oriented forest $\phi$ on $\XX$ is a
collection of rooted trees of $G$, oriented from their leaves towards  their root. 
A spanning oriented forest (s.o.f.) on $\XX$ is an oriented forest which exhausts
the points in $\XX$. The set of roots of a spanning oriented forest $\phi$ is denoted by $\rho(\phi)$. 
 
We introduce now a real parameter $q >0$, and associate to each oriented forest a weight
\begin{equation}
 \label{poidsforet.def}
 w_q(\phi) := q^{|\rho(\phi)|} \prod_{e \in \phi} w(e) \, .
\end{equation}
These weights can be renormalized to define a probability measure on the set of spanning oriented forest,
\begin{equation}
 \label{probaforet.def}
 \pi_q(\phi) := \frac{ w_q(\phi) }{Z(q)} \, ,
 \end{equation}
where the partition function $Z(q)$ is given by
\begin{equation}
 \label{partition.def}
 Z(q):= \sum_{\phi \mbox{ s.o.f.}}  w_q(\phi) \, .
\end{equation}

We can sample from $\pi_q$ by using Wilson's algorithm (\cite{Wi}, \cite{PW})
which can be described as follows.
Let $\Phi_c$ be the current state, an oriented forest, of the spanning oriented forest being constructed. 
At the beginning, $\Phi_c $ has no nodes or edges.
While $\Phi_c$ is not spanning, i.e., while there is a vertex in $\XX$ which is not in the vertex set $V(\Phi_c)$ of $\Phi_c$, perform the following steps:
\begin{itemize}
\item Choose a point $x$ in ${\mathcal X} \setminus V(\Phi_c)$, in any deterministic or random way.
\item  Let evolve the Markov process $(X(t), t \geq 0)$ from $x$, and stop it 
	at $T_q \wedge H_{V(\Phi_c)}$ with $T_q$ an independent exponential time  of parameter $q$
	and $H_{V(\Phi_c)}$ the hitting time of $V(\Phi_c)$.
\item Erase the loops of the trajectory drawn by $X$ to obtain a self-avoiding path $C$
	starting from $x$ and oriented towards its end-point.
\item Add $C$ to $\Phi_c$.
\end{itemize}
Each  iteration of the ``while loop'' stopped by the exponential time, gives birth to another tree. 
Wilson's algorithm is not only a way to sample $\pi_q$, it is also a powerful tool to study it.
The main strength of this algorithm is the total freedom one has in choosing the starting points $x$'s
of~$X$.

 In the sequel, $\Phi$ will denote a random variable defined on some probability space 
 $(\Omega_f, \AA_f, \P_q)$, having distribution $\pi_q$. The corresponding expectation will be denoted 
 by $\E_q$.  We will often work with two independent sources of randomness: the Markov 
 process $X$, and the random forest  $\Phi$. Integration with respect to $X$ starting from $x$ will
 be denoted by $P_x$ and $E_x$. When $X$ is started with an initial measure $\pi$, we will use 
 the notations $P_{\pi}$ and $E_{\pi}$. When we integrate over both randomness, we will use 
 the notations $\E_{x,q}, \E_{\pi,q}$ and  $\P_{x,q}, \P_{\pi,q}$. 
 The random forest $\Phi$ defines a partition of $\XX$, two points being in the same set of the partition
if they belong to the same tree. This partition will be denoted by $\AA(\Phi)$. A point $x \in \XX$ being
fixed, $t_x$ is the tree of $\Phi$ containing $x$, $\rho_x$ its root, and $A(x)$ the unique element of $\AA(\Phi)$
containing $x$.

A theorem of Kirchhoff \cite{Ki} gives in this context that
\begin{equation}
\label{part.exp}
Z(q) = \det (q \,  \Id - \LL)
	=  \prod_{j < n} (q + \lambda_j) ,
\end{equation}
and this implies (see for example \cite{LA} for more details,
a proof of \eqref{part.exp} and the following proposition):
\begin{prop}\label{nombre.prop}
   For all $k \in \acc{0, \cdots, n}$, 
 \[ \P_q\cro{\abs{\rho(\Phi)}=k} =
  \sum_{\begin{matrix}\scriptstyle{ J \subset \acc{0, \cdots, n-1}}\\[-2pt]  \scriptstyle{  \abs{J}=k}  \end{matrix} }
\,  \prod_{j \in J} \frac{q}{q+\lambda_j} \,\,   \prod_{j \notin J} \frac{\lambda_j}{q+\lambda_j} \, . 
 \]
Otherwise stated, the number of roots has the same law as $\sum_{j=0}^{n-1} B_j$ where 
$B_0, \cdots, B_{n-1}$ are independent, $B_j$ having Bernoulli distribution with parameter $\frac{q}{q+\lambda_j}$.
\end{prop} 

\subsection{Main results: approximate and exact solutions of the intertwining equations}\label{grey}
\subsubsection{Approximate solution of $\Lambda P = \bar P\Lambda$}\label{yellow}
Assume that we sampled $\Phi$ from $\pi_q$ for some parameter $q > 0$.
For $q' > 0$ we then set
\begin{itemize}
\item $\bX := \rho(\Phi)$;
\item For any $\bx \in \bX$, $\nu_{\bx}(\cdot) := K_{q'}(\bx,\cdot)$ (cf. equation \eqref{Kq.def}), i.e. 
   $\Lambda= K_{q'} |_{\bX \times \XX}$;
\item $\bP(\bx,\by):=P_{\bx}\cro{X(H_{\bX}^+)=\by}$ with, for any $A \subset {\mathcal X}$,
	\begin{equation} \label{return.def}
		H^+_A := \inf \acc{t \geq \tau_1, X(t) \in A } \, . 
	 \end{equation}
	$H^+_A$ is in other words the return time in $A$, and $\bar P$ is the (irreducible and reversible) Markovian kernel
	associated with the trace chain of $X$ on $\bar{\mathcal X}$.
\end{itemize}
Here $\bX$ is a random subset of ${\mathcal X}$, and so is its cardinality.
If we want to keep approximately $m$ points from $\XX$, we have  to ensure that 
\begin{equation}
\label{Relationmq.eq}
\E_q\cro{\abs{\bX}} = \sum_{i=0}^{n-1} \frac{q}{q+ \lambda_i} \approx m\, .
\end{equation}
This can be obtained, starting from any $q$ to sample $\Phi$, by updating $q$ according to
$q \leftarrow q \times  m / |\rho(\Phi)|$ before re-sampling $\Phi$ and going so up to getting
a satisfactory number of roots (see~\cite{LA} for more details).

Let us now define fore each $j \in \acc{0,\cdots,n-1}$, 
\[ p_j := \frac{q}{q+\lambda_j}\,, \quad  p'_j := \frac{q'}{q'+\lambda_j}    \, .
\]
and denote by $d_{TV}$ 
the total variation distance:
if $\nu$ and $\nu'$ are two probability measures on $\XX$, 
\[ d_{TV}(\nu, \nu') = \frac 1 2 \sum_{x \in \XX} \abs{\nu(x)-\nu'(x)} \, .
\]
\begin{theo} \label{cornet}
For all $m \in \acc{1,\cdots,n}$,
\begin{equation}
\label{TVDFm.eq}
 \E_q \cro{\sum_{\bx \in \bX} d_{TV}(\Lambda P(\bx,\cdot), \bP \Lambda(\bx, \cdot))  \biggm| {\abs{\bX}=m}  } 
\leq \frac{q' (n-m)}{\alpha} \,,
 \end{equation}
and 
\begin{equation}
\label{TVDF.eq}
 \E_q \cro{\sum_{\bx \in \bX} d_{TV}(\Lambda P(\bx,\cdot), \bP \Lambda(\bx, \cdot))} 
\leq  \frac{q'}{\alpha} \sum_{i=1}^{n-1} \frac{ \lambda_i}{q + \lambda_i} \, .
 \end{equation}
In addition, with
\[ S_n := \sum_{j=1}^{n-1} p'^2_j (1-p_j)^2 \, \, ; \, \, 
T_n:= \sum_{j=1}^{n-1} \frac{p_j^2}{p'^2_j } \, \, ; \, \,  
V_n =  \sum_{j=1}^{n-1} p_j (1-p_j) 
\,, 
\]
it holds
\begin{equation} \label{flatDF.eq}
	\E_q\cro{ \SS(\Lambda) \biggm| {\abs{\bX}=m}}
	\leq \frac{
		\min\left\{
			\sqrt{1 + \sqrt{\frac{T_n}{S_n}}} \; \exp\pare{\sqrt{S_n T_n}- V_n} ;
			\sqrt{1 + T_n} \exp\pare{\frac{\left(1 + S_n T_n)\right)}{2} - V_n}
		\right\}
	}{\P_q\cro{|\bar{\mathcal X}| = m}}
\end{equation}
for any $m \in \acc{1,\cdots,n}$.
\end{theo}
\noindent{\it Proof:} See Section~\ref{jazz}.

\noindent {\it Comment:}
Our upper bounds depend on ${\mathcal L}$ through its spectrum only. They show that if there is a gap in this spectrum
---that is if for some $1 < m < n$ it holds $\lambda_{m - 1} \ll \lambda_m$--- then we can have asymptotically exact
solutions with small squeezing by choosing $\lambda_{m - 1} \ll q \ll q' \ll \lambda_m$. We then have indeed
$q' \ll \alpha$ since $\lambda_m \leq 2\alpha$ and $p_j \sim p_j' \sim 1$ for $j < m$,
while $p_j \ll p_j' \ll 1$ for $j \geq m$.
We can then have a vanishing error in the approximation, see~\eqref{TVDF.eq}.
In addition we can have $V_n \ll 1$, ${S_n} \ll 1$, $T_n \sim m - 1$, $\P_q[|\bar{\mathcal X}| = m] \sim 1$
(recall Proposition~\ref{nombre.prop})
and an upper bound on the mean value of ${\mathcal S}(\Lambda)$ that goes like $\sqrt m$.
This upper bound has to be compared with the lower bounds of Proposition~\ref{flatness.prop}, i.e. with $1$ if
we have asymptotic solutions of intertwining equations. 
But for some simple low temperature metastable systems leading to such a gap in the spectrum
we will have ${\mathcal S}(\Lambda) \sim 1$: our upper bound is not optimal.

\subsubsection{Approximate solutions of $\Lambda K_{q'} = \bar P\Lambda$}
Assume once again that we sampled $\Phi$ from $\pi_q$ for some parameter $q > 0$.
But let us modify our choices for $\bar{\mathcal X}$, $\Lambda$ and $\bar P$,
by using this time the partition $\AA(\Phi)$. Set:
 \begin{itemize}
 \item $\bX:=\rho(\Phi)$ (one could rather think that $\bar{\mathcal X}$ is the set
	of the different pieces forming the partition ${\mathcal A}(\Phi)$ but the notation
	will be simpler by using the set of roots, which obviously is in one to one correspondence 
	through the map $A : \bar x \in \rho(\Phi) \mapsto A(\bar x)$);
 \item for any $\bx \in \bX$, $\nu_{\bx}(\cdot) := \mu_{A(\bx)}(\cdot)$, with, for any $A \subset {\mathcal X}$,
	$\mu_A$ being defined by the probability $\mu$ conditioned to $A$: $\mu_A := \mu(\cdot | A)$;
 \item for any $\bx,\by \in \bX$, $\bP(\bx,\by):=  P_{\mu_{A(\bx)}}\cro{X(T_{q'}) \in A(\by)}$, with $T_{q'}$ being
	as previously an exponential random variable of parameter $q'$ that is independent from $X$.
	Irreducibility and reversibility of $\bar P$ are then inherited from those of $P$.
 \end{itemize}
 It follows from Proposition \ref{flatness.prop} that the squeezing of $\acc{\nu_{\bx}, \bx \in \bX}$ is minimal and equal 
 to one. 
 
To bound the distance between $\Lambda K_{q'}$ and $\bP \Lambda$, we introduce another random
forest $\Phi'$ distributed as $\pi_{q'}$ and independent of $\Phi$ and $X$. For any $x \in \XX$,
$t'_x$ is the tree containing $x$ in $\Phi'$, $\rho'_x$ its root, $A'(x)$ the unique element
of $\AA(\Phi')$ containing $x$, and $\Gamma'_x$ is the path  going from $x$ to $\rho'_x$ in 
$\Phi'$. By Wilson algorithm started at~$x$, $\Gamma'_x$ is the trajectory of a loop-erased random walk
started from $x$ and stopped at an exponential time $T_{q'}$. We denote by $|\Gamma'_x|$ its length,
that is the number of edges to be crossed in $\Phi'$ to go from $x$ to $\rho'_x$.

\begin{theo}
\label{TVmeta.theo} Let $p \geq 1$, and $p^*$ its conjugate exponent, so that $\frac 1 p + \frac 1{p^*} = 1$. 
\[ \E_q\cro{\sum_{\bx \in \bX} d_{TV}(\Lambda K_{q'}(\bx, \cdot), \bP \Lambda(\bx, \cdot))}
\leq \pare{\E_q\cro{\abs{\rho(\Phi)}}}^{1/p}
\pare{\frac{q'}{q} \sum_{x \in \XX} \E_{q'}\cro{|\Gamma'_x|}}^{1/p^*} 
.
\] 
\end{theo}
\noindent {\it Proof:} See Section~\ref{dizzy}.

\noindent {\it Comment:}
Note that
$$
	q'{\mathbb E}_{q'}\left[|\Gamma'_x|\right] = \alpha\frac{{\mathbb E}_{q'}\left[|\Gamma'_x|\right]}{\alpha / q'}
$$
is, up to the factor $\alpha$, 
the ratio between the mean number of steps of the loop-erased random walk and the mean number of steps 
of the simple random walk up to time $T_{q'}$, that is the time fraction spent outside loops up to time $T_{q'}$.
As a consequence ``the more recurrent is $X$ on time scale $1 / q'\,$'', the smaller is this ratio.

\subsubsection{Exact solutions of $\Lambda K_{q'} = \bar P\Lambda$}
We finally modify the previous random measures $\mu_{A(\bar x)}$
to build exact solution of Equation~\eqref{meta.eq} for $q'$ small enough.
We will use to this end a result due to Micchelli and Willoughby \cite{MW}:
for any $m > 0$ 
$$
	MW_m := \prod_{j \geq m} \frac1{\lambda_j}\left({\mathcal L} + \lambda_j{\rm Id}\right)
$$ 
is a Markovian kernel (one can see \cite{LA} for a probabilistic insight into the proof of this result).
Assume then that we sampled $\Phi$ from $\pi_q$ for some parameter $q > 0$,
let us keep $\bar{\mathcal X} = \rho(\Phi)$, but let us now set 
$$
	\nu_{\bar x} = \mu_{A(\bar x)} MW_m, 
	\qquad \bar x \in \bar{\mathcal X},
$$
with $m = |\bar{\mathcal X}|$.

\begin{theo} \label{cup}
	If the $\nu_{\bar x}$ have finite squeezing, then for $q'$ small enough,
	the $\nu_{\bar x} K_{q'}$ are in the convex hull of the $\nu_{\bar x}$.
\end{theo}
\noindent {\it Proof:} See Section~\ref{chopin}.

\noindent {\it Comment:}
Since we do not give quantitative bounds on how small $q'$ has to be for the thesis to hold,
and we do not bound the squeezing of these $\nu_{\bar x}$,
Theorem~\ref{cup} is almost a trivial result.
However the proof we will give suggests that the $\nu_{\bar x}$
are natural candidates for not too squeezed solution associated with some non very small $q'$.
It will also give further motivation to use squeezing to measure joint overlap:
we actually got to our squeezing definition by looking for quantitative bounds for this theorem.

\subsection{Signal processing and metastability}\label{violet}
\subsubsection{Pyramidal algorithms in signal processing}\label{green}
Our motivations for the previous results come both from signal processing and metastability studies.
First we are interested in extending classical pyramidal algorithms of signal processing on the discrete torus 
$${\mathcal X} = {\mathcal X}_0 = {\mathbb Z}_n = {\mathbb Z} / n{\mathbb Z}$$ to the case of signals on generic edge-weighted graphs.
Such algorithms are used for example to analyze or compress a given signal
$$ f{=}f_0 : {\mathcal X}_0 \rightarrow {\mathbb R} $$
through {\em filtering} and {\em subsampling} operations.
A {\em filter} is a linear operator which is diagonal in the same base as the discrete Laplacian ${\mathcal L}$.
A {\em low-pass filter} $K$
has eigenvalues of order 1 for low frequency modes,
i.e., eigenvectors that are associated with small eigenvalues of $-{\mathcal L}$,
and it has small eigenvalues for high frequency modes,
i.e., eigenvectors that are associated with large eigenvalues of $-{\mathcal L}$.
A pyramidal algorithm first computes $m = n/2$ {\em approximation coefficients} by
\begin{itemize}
\item computing a low-pass filtered version $Kf$ of the original signal $f$,
\item subsampling $Kf$ by keeping one in each two
	of its $n$ values, those in some ${\mathcal X}_1 = \bar{\mathcal X} \subset {\mathcal X}$,
	for example the $n / 2$ values in the even sites of ${\mathbb Z}_n$.
\end{itemize}
In doing so it defines a function $$\bar f : \bar x \in \bar{\mathcal X} \mapsto Kf(\bar x) \in {\mathbb R}$$
that can naturally be seen as a signal $f_1 : {\mathbb Z}_{n / 2} \rightarrow {\mathbb R}$
on a twice smaller torus.
It then computes an {\em approximation} $\tilde f$ of $f$ on ${\mathcal X}$ as a function of the approximation coefficients,
and a {\em detail function} $\tilde g = f - \tilde f$, which in turn can be encoded into $n - m$ {\em detail coefficients}.
Wavelet decomposition algorithms are of this kind.
It then applies a similar treatment to $f_1$, to define $f_2$, then $f_3$, \dots\ up to reaching a simple signal defined
on a small torus made of a few points only.
The reason why this can be useful for compression is that, for well chosen filters,
many of the detail coefficients obtained at the different levels are very small or negligible for a large class of smooth signals $f$.
And one just has to store the few non-negligible detail coefficients together with the coarsest approximation's coefficients
to reconstruct a good approximation of the original signal $f$. 
The point is then to find ``good'' filters, i.e. ``good'' $\varphi_{\bar x}$ in $\ell_2(\mu)$ (in this case $\mu$ is the uniform measure 
on ${\mathcal X}$, the reversible measure of the simple random walk associated with the discrete Laplacian)
so that, for all $f \in \ell_2(\mu)$, $$\langle\varphi_{\bar x}, f\rangle = Kf(\bar x).$$
And a basic requirement for good filters is that, for each $\bar x$, $\varphi_{\bar x}$ is {\em localized} around $\bar x$.
Even though the dual $\nu_{\bar x} = \varphi_{\bar x}^*$ are usually signed measures and not measures,
this is the reason why we want to think of the computation of the approximation coefficients $\bar f(\bar x) = \langle\nu_{\bar x} | f\rangle$
as computation of {\em local means}.
Since, $K$ being a low-pass filter, $\varphi_{\bar x}$ needs also to be ``localized'' in Fourier space 
(written in the diagonalizing basis of ${\mathcal L}$, it must have small coefficients on high-frequency modes).
Thus the difficulty comes from Heisenberg principle, which roughly says that no function $\varphi_{\bar x}$
can be well localized both in Fourier space and around $\bar x$.
Part of the art of wavelet design
lies in the ability to make a good compromise with Heisenberg principle
(see for example Chapter 7 in~\cite{VKG} for more details
on this point).

When moving to the case of signal processing for generic edge-weighted graph,
there are three main issues one has immediately to address
to build pyramidal algorithms:
\begin{enumerate}
\item[(Q1)] What kind of subsampling should one use? What could ``one every second node'' mean?
\item[(Q2)] Which kind of filter should one use? How to compute local means?
\item[(Q3)] On which (weighted) graph should the approximation coefficients $\bar f(\bar x)$ be defined to iterate the procedure?
\end{enumerate}
On a general weighted finite graph $G=(\XX,\EE,w)$ with $\EE=\{(x,y):x\neq y,w(x,y)>0\}$
none of these questions has a canonical answer.
Several attemps to tackle these issues and to generalize wavelet constructions have been proposed:
see \cite{SCHU} for a review on this subject and \cite{HAM} for one of the most popular method.
To our knowledge our proposition is the first one based on the solution of intertwining equations.
To motivate it let us first make the connection with metastability studies.

\subsubsection{Metastability and intertwining} \label{blue}
We first note that, since $\bar f$ should represent a coarse-grained version of $f$,
it is natural to think that the weights $\bar w(\bar x, \bar y)$ to be build
to answer to question (Q3) should be such that the obtained graph is itself
a coarse-grained version of $G$.
Since the law of the Markov process $X$ with generator ${\mathcal L}$
defined through the weights $w(x, y)$ completely characterized these weights,
we would like the associated process $\bar X$ to be a ``coarse-grained version of $X$''.
This is what we are used to build in metastability studies, possibly by seeing
$\bar X$ as a measure-valued process on a small state space, these measures being
probability measures on the large state space~${\mathcal X}$, on which $X$ is defined.
For example, when we want to describe the crystallisation of a slightly supersaturated vapor,
we can do it in the following way.
Vapor and crystal are defined by probability measures
concentrated on very different parts of a very large state space.
On this space a Markov process describing the temporal evolution of a microscopic configuration
evolves,
and this Markovian evolution has to be ``macroscopically captured''
by a new two-state Markov process evolving from gas (a probability measure on the large state space)
to crystal (another probability measure on the same space
almost non-overlapping with the previous one).
And this evolution is such that the gas should appear
as a {\em local equilibrium} left only to reach a more stable crystalline equilibrium.
This is usually done in some asymptotic regime (e.g. large volume or low temperature asymptotic)
and we refer to~\cite{OV} and~\cite{BdH} for mathematical accounts on the subject.

But we are here outside any asymptotic regime: we are given a finite graph $({\mathcal X}, w)$
or a Markov process $X$ and we want to define a finite coarse-grained version of this graph
and Markov process, $(\bar{\mathcal X}, \bar w)$ and $\bar X$. 
Solving intertwining equation $\Lambda P = \bar P\Lambda$ provides a way to do so.
When the size of $\bar P$ is smaller than the size of $P$,
equations~\eqref{banzai} suggest that the evolution of $X$ can be roughly described through that of $\bar X$:
from state or local equilibrium $\nu_{\bar x}$ the process $X$ evolves towards a new state
or local equilibrium~$\nu_{\bar y}$ 
which is chosen according to the Markovian kernel $\bar P$.
This can be turned into a rigorous mathematical statement by the following proposition
(recall that in our notation $\hat X$ is the discrete time Markov chain with kernel $P$).
\begin{prop} \label{poireau}
If Equation~\eqref{DF.eq} holds, i.e. Equation~\eqref{banzai} is in force for each $\bar x$ in $\bar{\mathcal X}$,
then, there is a filtration ${\mathcal F}$ for which $\hat X$ is ${\mathcal F}$-adapted and such that,
for each $\bar x$ in $\bar{\mathcal X}$ there is a ${\mathcal F}$-stopping time $T_{\bar x}$ and a random variable
$\bar Y_{\bar x}$ with value in $\bar{\mathcal X} \setminus \{\bar x\}$ such that
\begin{enumerate}
\item $T_{\bar x}$ is geometric with parameter $1 - \bar P(\bar x, \bar x)$;
\item $\nu_{\bar x}$ is stationary up to $T_{\bar x}$, i.e., for all $t \geq 0$,
	\begin{equation}\label{mucca}
		P_{\nu_{\bar x}}\left(\hat X(t) = \cdot \bigm| t < T_{\bar x}\right) = \nu_{\bar x}
		\,;
	\end{equation}
\item $P_{\nu_{\bar x}}\left(\bar Y_{\bar x} = \bar y\right) = \frac{\bar P(\bar x, \bar y)}{1 - \bar P(\bar x, \bar x)}$
	for all $\bar y$ in $\bar{\mathcal X}\setminus\{\bar x\}$;
\item $P_{\nu_{\bar x}}\left(\hat X(T_{\bar x}) = \cdot \bigm| \bar Y_{\bar x} = \bar y\right) = \nu_{\bar y}(\cdot)$;
\item $\left(\bar Y_{\bar x}, \hat X(T_{\bar x})\right)$ and $T_{\bar x}$ are independent.
\end{enumerate}
\end{prop}
\noindent This is a partial rewriting of Section~2.4 of~\cite{DF} in the spirit of~\cite{Mi}.
We give a proof of this proposition in appendix.

As far as metastability is concerned, a possibly more natural approach is to look at solution
of Equation~\eqref{meta.eq} rather than Equation~\eqref{DF.eq}:
it is on a ``long'' time scale~$1 / q'$ that one is looking at a coarse-grained Markovian version
of~$X$. Whatever the equation, \eqref{DF.eq} or~\eqref{meta.eq}, we are looking at,
we still want solutions $\nu_{\bar x}$ that are localized in well distinct part of the state space:
we are looking for little squeezed solutions.

Coming back to signal processing and the previously raised question (Q3),
a possible answer to it (and to question (Q2) also)
is then to find $\bar P$ (and $\Lambda$) such that~\eqref{DF.eq}
holds with $m = |\bar{\mathcal X}| < n$, with $m$ and $n$ of the same order.
Since we want to keep our irreducibility and reversibility hypothesis 
to deal with $\bar P$ at the next level in the pyramidal algorithm in the same way we deal with $P$,
we are mainly interested in irreducible and reversible $\bar P$.
In metastability studies we are often interested in cases where $m$ is very small with respect to $n$.
However if one implements as proposed the complete pyramidal algorithm,
one will solve at the same time intertwining equations with very different $m$ and $n$
by transitivity of the coarse-graining procedure.

\subsubsection{Heisenberg principle, approximate solutions and further work}
There is actually at least a fourth question
without canonical answer
that arises
when going from classical pyramidal or wavelet algorithms
to signal processing for generic weighted graphs:
what is a ``Heisenberg principle'' limiting the localisation of our $\nu_{\bar x}$?
We do not have an answer to this question, but, although we explained why we are interested
in localized, non-overlapping, little squeezed solutions of the intertwining equations,
we will see in the next section that exact solutions of intertwining equations
are localized in Fourier space, just as the $\varphi_{\bar x}$ should be in classical algorithms.
This is the main difficulty faced by the present approach
and this is one of the two reasons why we turned to approximate solutions of intertwining equations.
We will also see in the next section that one needs a detailed knowledge of the spectrum and the eigenvectors
of the Laplacian ${\mathcal L}$ to build exact solutions of intertwining equations.
From an algorithmic point of view this can be very costly,
and this is the other reason why we turned to approximate solutions.

In a forthcoming paper we will analyse the full pyramidal algorithm,
including a wavelet basis construction,
rather than simply focusing on intertwining equations of a one-step reduction.
But we are still looking for a generalised Heisenberg principle that could serve
as a guideline for similar constructions.
And our results suggest that such a Heisenberg principle should degenerate
in presence of a gap in the spectrum (see~\ref{yellow}). 
Let us now conclude this first section by giving some heuristics
and explaining how random forests enter into the game.

\subsection{Well distributed points and determinantal processes}\label{pink}
One of us spent many years in looking for a practical approach to metastability
through intertwining relations without being able to go beyond the first simple observations
of Section~\ref{campi}. Some progress were eventually achieved,
only when the connection with signal processing was made.
We explained the connection between pyramidal algorithms, metastability studies and intertwining equations
by proposing mathematical formulations for two of the three raised questions (see~\ref{green} and~\ref{blue}).
It turns out that the first question is a much simpler one, for which a random solution is proposed in~\cite{LA}.
This solution then suggests a way to answer the last two questions.

The subsampling issue is that of finding $m$ points, a fraction of $n$, that are in some sense
well distributed in ${\mathcal X}$.
Let us denote, for any subset~$A$ of~$\XX$, by~$H_A$ and $H_A^+$ the hitting time of
and the return time to~$A$ for the process~$X$:
\[
H_A := \inf \acc{t \geq 0, X(t) \in A },
\]
\[
H_A^+ := \inf \acc{t \geq \tau_1, X(t) \in A },
\]
with $\tau_1$ the first time of the Poisson process
that links $\hat X$ with $X$ (see Section~\ref{genova}).
For each $x$ in~${\mathcal X}$ the mean hitting time $E_x[H_{\rho(\Phi)}]$ is a random variable,
since so are $\Phi$ and $\rho(\Phi)$.
And it turns out that its expected value, with or without conditioning on the size of $\rho(\Phi)$,
does not depend on $x$. In this sense the 
roots of the random forest are ``well spread'' on $\XX$. 
More precisely we have (see~\cite{LA}):
\begin{prop} \label{hitting.prop}
For any $x \in \XX$ and $m \in \acc{1, \cdots, n}$
it holds
\begin{equation}
\E_{x,q}\cro{H_{\rho(\Phi)}} = \frac{\P_q\cro{\abs{\rho(\Phi)}>1}}{q} \,;
\end{equation}
\begin{equation}
\E_{x,q}\cro{H_{\rho(\Phi)}\bigm|{ \abs{\rho(\Phi)}=m}} 
= \frac{\P_q\cro{\abs{\rho(\Phi)}=m+1}}{q \P_q\cro{\abs{\rho(\Phi)}=m}} \, ;
\end{equation}
\begin{equation}\label{luna}
\E_q\cro{\frac 1m \sum_{\bar x \in \rho(\Phi)} E_{\bar x}\cro{H^+_{\rho(\Phi)} \bigm| \abs{\rho(\Phi)} = m}} = \frac{n}{\alpha m} \,.
\end{equation}
\end{prop}
\noindent
This suggest to take $\bar{\mathcal X} = \rho(\Phi)$.

As a consequence of Burton and Pemantle's transfer current Theorem,
$\rho(\Phi)$ is a determinantal process on ${\mathcal X}$, and its kernel
is $K_q = q(q{\rm Id} - {\mathcal L})^{-1}$ (see~\cite{LA}):
\def \det {{\rm det}} 
\begin{prop} \label{racines.prop}
 For any subset  $A$ of $\XX$,
 \[ \P_q( A \subset \rho(\Phi)) = \det_{A}(K_q) \, , 
 \]
 where $\det_A$ applied to some matrix is the minor defined by the rows and columns corresponding to~$A$.
\end{prop}
\noindent
By using reversibility,
one can see that the determinant of $(K_q(x, y))_{x, y \in A}$ is, 
up to a multiplicative factor $\prod_{x \in A} \mu(x)$,
the Gram matrix of the distributions  
$\bigr(P_x\bigl(X(\tilde T_q) = \cdot\bigr), x\in A\big)$,
with $\tilde T_q$ the square of an independent centered Gaussian variable
with variance $1 /(2q)$
(in such a way that the sum
of two independent copies of $\tilde T_q$ 
has the same law as $T_q$).
This means that a family of nodes $\bar x$ is unlikely to be part of $\rho(\Phi)$
if the volume of the parallelepiped formed by
these distributions is small.
It suggests that the distributions
$\bigl(P_{\bar x}\bigl(X(\tilde T_q) = \cdot\bigr), \bar x\in\rho(\Phi)\bigr)$
are typically little squeezed
and so should be the distributions 
$\bigl(K_q(\bar x, \cdot), \bar x\in\rho(\Phi)\bigl)$,
which are easier to deal with.
To have a trade-off between squeezing and approximation error in intertwining equations,
it will be convenient to introduce a second parameter $q' > 0$ and set $\nu_{\bar x} = K_{q'}(\bar x, \cdot)$
for $\bar x$ in $\rho(\Phi)$.
At this point the choice made for $\bar P$ in~\ref{yellow}
may be the most natural one.

Finally, when dealing with metastability issues, 
building local equilibria $\nu_{\bar x}$ from single ``microscopic configurations'' ${\bar x}$ in $\rho(\Phi)$
seems rather unnatural.
In our previous example, no special microscopic configuration should play a role in defining 
what a metastable vapor should be.
One should better look for larger structures associated with $\Phi$, like ${\mathcal A}(\Phi)$
rather than $\rho(\Phi)$.
Then, in view of the following proposition from~\cite{LA},
the unsqueezed measures $\mu_{A(\bar x)}$ appear to be natural candidates
for giving approximate solutions of~\eqref{meta.eq}:
\begin{prop} 
\label{rootsgivenpartition.prop}
Conditional law of the roots, given the partition. \\
Let $m$ be fixed, and $A_1$, \dots, $A_m$ be a partition of $\XX$. For any $x_1 \in A_1, \cdots, x_m \in A_m$, 
\begin{equation}
\P_q\cro{\rho(\Phi)=\acc{x_1,\cdots,x_m} \bigm| {\AA(\phi)=(A_1,\cdots,A_m)}} = \prod_{i=1}^m \mu_{A_i}(x_i) \, ,
\end{equation}
where $\mu_A$ is the invariant measure $\mu$ conditioned to $A$ ($\mu_A(B)=\mu(A \cap B)/\mu(A)$). Hence,
given the partition, the roots are independent, and distributed according to the invariant measure. 
\end{prop}

\section{Preliminary results}\label{campi}
\subsection{Proof of Proposition~\ref{flatness.prop}}
If $\Gamma$ is not invertible,  points (1) and (2) are obviously true. We assume therefore that $\Gamma$ is invertible. 
Let $\tilde{\Lambda}:=\Gamma^{-1} {\Lambda}$, and let $(\tilde{\nu}_{\bx}, \bx \in \bX)$ be the row vectors
 of $\tilde{\Lambda}$. Note that 
 \[ \tilde{\Lambda} D(1/\mu) {\Lambda}^t = \Gamma^{-1} {\Lambda} D(1/\mu) {\Lambda}^t= \Gamma^{-1} \Gamma = \Id 
 \, . 
 \]
 \[ \tilde{\Lambda} D(1/\mu) \tilde{\Lambda}^t = \Gamma^{-1} {\Lambda} D(1/\mu) {\Lambda}^t \Gamma^{-1} =
 \Gamma^{-1} 
 \, .
 \]
 Hence, for all $\bx, \by \in \XX$, $\bras{\tilde{\nu}_{\bx}, \nu_{\by}}= \delta_{\bx \by}$ and 
 $\nor{\tilde{\nu}_{\bx}}^{*2}=  (\Gamma^{-1} )(\bx,\bx)$. 

\begin{enumerate}
\item We have $\SS({\Lambda})^2= \sum_{\bx \in \bX} \nor{\tilde{\nu}_{\bx}}^{*2} \geq 
 \sum_{\bx \in \bX} \frac{\bra{\tilde{\nu}_{\bx}, \nu_{\bx}}^{*2}}{\nor{\nu_{\bx}}^{*2}}
 =  \sum_{\bx \in \bX} \frac{1}{\nor{\nu_{\bx}}^{*2}}$. 
 Assume now that the $\nu_{\bx}$'s, $\bx \in \bX$ are orthogonal. $\Gamma= \mbox{diag}(\nor{\nu_{\bx}}^{*2})$, so that 
$\Tr( \Gamma^{-1} )=  \sum_{\bx \in \bX}  \frac{1}{\nor{\nu_{\bx}}^{*2}}$. In the opposite direction, 
assume instead that $\Tr( \Gamma^{-1} )=  \sum_{\bx \in \bX}  \frac{1}{\nor{\nu_{\bx}}^{*2}}$. Then for any 
  $\bx \in \bX$, $\abs{\bras{\tilde{\nu}_{\bx}, \nu_{\bx}}} = \nors{\tilde{\nu}_{\bx}}  \nors{\nu_{\bx}}$. This implies
  that for all $\bx \in \bX$,  there exists a real number $\alpha(\bx) \neq 0$ such that $\tilde{\nu}_{\bx}=
  \alpha(\bx) \nu_{\bx}$. Taking the scalar product with 
  $\nu_{\by}$ leads to $\delta_{\bx \by} = \bras{\tilde{\nu}_{\bx} , \nu_{\by}}= \alpha(\bx) \bras{\nu_{\bx},\nu_{\by}}$. Hence $(\nu_{\bx}, \bx \in \bX)$ are orthogonal. 

 \item  
  Let us write $\mu$ as a convex combination of the $(\nu_{\bx}, \bx \in \bX)$: 
  \[ \mu= \sum_{\bx \in \bX} \alpha(\bx) \nu_{\bx}, \qquad \alpha(\bx) \geq 0, \qquad \sum _{\bx \in \bX} \alpha(\bx) =1.
  \]
  Note that for any probability measure $\nu$, $\bras{\mu,\nu}= \sum_{x \in \XX} \mu(x) \nu(x) / \mu(x) =1$.
 As a special case, for any $\by \in \bX$,
\begin{equation}
\label{coefmu.ineq}
 1=\bras{\mu,\nu_{\by}} = \sum_{\bx \in \bX} \alpha(\bx) \bras{ \nu_{\bx},\nu_{\by}} \geq
  \alpha(\by) \nor{\nu_{\by}}^{*2} \, .
  \end{equation}
 By point (1), we deduce that 
 \[ \SS({\Lambda})^2 \geq \sum_{\bx \in \bX} \frac{1}{\nor{\nu_{\bx}}^{*2}}  \geq   \sum_{\bx \in \bX}  \alpha(\bx) =
 1 \, .
 \]
 Equality holds if and only if \eqref{coefmu.ineq} and  \eqref{Trace.ineq} are equalities. By point (1), this implies that the 
 $(\nu_{\bx}, \bx \in \bX)$ are orthogonal.
 In the opposite direction,
 when the  $(\nu_{\bx}, \bx \in \bX)$ are orthogonal, \eqref{coefmu.ineq} and  \eqref{Trace.ineq} are equalities, and  $\SS({\Lambda}) = 1$.   
 \end{enumerate}

\subsection{Elementary observations on intertwining equations}\label{impianto}
Consider Equation~\eqref{DF.eq} for any reversible and irreducible stochastic kernel $P$,
and assume an $m\times n$ rectangular stochastic matrix $\Lambda = (\Lambda(\bar x, x))_{\bar x \in \bar{\mathcal X}, x \in {\mathcal X}}$
to be a solution for some $\bar P$ with $m \leq n$.
Let us write $(\theta_j)_{j < n} = (1 - \lambda_j / \alpha)_{j < n}$ for the $n$ eigenvalues of $P$ in decreasing order:
$$
	1 = \theta_0 > \theta_1 \geq \cdots \geq \theta_{n - 1} \geq -1.
$$
We also set $[n] = \{0, 1, 2, \dots, n-1\}$, call $\mu$ the reversible measure of $P$,
and write $\nu_{\bar x} = \Lambda(\bar x, \cdot)$ for the rows of $\Lambda$.
\begin{lem}
	If $\Lambda$ is non-degenerate, i.e., if $\Lambda$ is of rank $m$,
	then there is an orthonormal basis of left eigenvectors $(\mu_j : 0 \leq j < n)$ of $P$ such that
	$$
		\mu_j P = \theta_j \mu_j, \qquad j < n,
	$$
	there is a subset $J$ of $[n]$ such that $0 \in J$ and $|J| = m$ 
	and there is an invertible matrix $C = (C(\bar x, j))_{\bar x \in \bar{\mathcal X}, j \in J}$
	with $C(\bar x, 0) = 1$ for all~$\bar x$ in~$\bar{\mathcal X}$,
	such that
	\begin{equation}\label{faya}
		\nu_{\bar x} = \sum_{j \in J} C(\bar x, j) \mu_j, 	\qquad \bar x \in \bar{\mathcal X},
	\end{equation}
	and		
	\begin{equation}\label{manuel}
		\bar P C(\cdot, j) = \theta_j C(\cdot, j), \qquad j \in J.
	\end{equation}
	In particular, the spectrum of $\bar P$ is contained in that of $P$,
	with eigenvalue multiplicities that do not exceed the corresponding ones for $P$.
\end{lem}
\begin{proof}
	Let $V$ be the subspace of $\ell_2^*(\mu)$ spanned by the $\nu_{\bar x}$.
	Since $\Lambda$ is non-degenerate, $V$ is of dimension $m$.
	Since $\Lambda P = \bar P \Lambda$,
	the $\nu_{\bar x}P$ are convex combinations of the $\nu_{\bar x}$
	and $V$ is stable by the self-adjoint operator~$P$.
	It follows that there is such an orthonormal basis of left eigenvectors $\mu_j$,
	with $\mu_0 = \mu$, a subset $J \subset [n]$ of size $m$,
	and an invertible matrix $C$ such that~\eqref{faya} holds.
	Since for $j > 0$ one has $\langle \mu, \mu_j\rangle^* = 0$,
	by computing the scalar product with~$\mu$ of both sides of equations~\eqref{faya},
	it follows that~$0$ belongs to $J$ and $C(\bar x, 0) = 1$
	for each $\bar x$.
	
	Now, applying $P$ on both sides of~\eqref{faya} we obtain
	$$
		\sum_{j \in J} \sum_{\bar y \in \bar{\mathcal X}} \bar P(\bar x, \bar y) C(\bar y, j) \mu_j
		= \sum_{j \in J} \theta_j C(\bar x, j) \mu_j,
		\qquad \bar x \in \bar{\mathcal X}.
	$$
	By identifying the decomposition coefficients in the basis of the $\mu_j$'s, this gives~\eqref{manuel}.
	Since the~$m$ column vectors $C(\cdot, j)$ are linearly independent, they form a basis of the functions on $\bar{\mathcal X}$.
	This is why equations~\eqref{manuel} completely describe the spectrum of $\bar P$ and we can conclude
	that the spectrum of $\bar P$ is contained in that of $P$ with the multiplicity constraint.
\end{proof}
The previous lemma shows on the one hand a localisation property in Fourier space of exact solutions
of intertwining equations: the $\nu_{\bar x}$ have to be with no component on $n - m$ eigenvectors
of the Laplacian ${\mathcal L}$ (see equations~\eqref{faya}).
On the other hand, it shows that finding exact solutions of intertwining equation 
implies to have a detailed knowledge of the eigenvectors of the Laplacian.

Conversely, it is now possible to describe all the non-degenerate solutions of the intertwining equations
in terms of, on the one hand, the eigenvectors and eigenvalues of $P$ and,
on the other hand, the set of diagonalizable stochastic matrices $\bar P$
with a given spectrum contained in that of $P$, and satisfying the multiplicity constraint.
Any right eigenvector basis $(C(\cdot, j) : j \in J)$
---satisfying~\eqref{manuel} and with $C(\cdot, 0) \equiv 1$---
of such a $\bar P$ will provide,
through equations~\eqref{faya} and possibly after rescaling, a non-degenerate solution of the intertwining equations.
The only delicate point to check is indeed the non-negativity of the~$\nu_{\bar x}$.
But if this fails, and since $\mu = \mu_0$ charges all points in ${\mathcal X}$, one just has to replace the $C(\cdot, j)$ for positive $j$ in $J$,
by some $\delta_j C(\cdot, j)$ for some small enough~$\delta_j$.

At this point we just have to give sufficient conditions
for the set of diagonalisable stochastic matrices
with a given spectrum 
to ensure that our intertwining equations do have solutions.
The next lemma shows that, if~$P$ has non-negative
eigenvalues, then we will find solutions
with $\bar{\mathcal X}$ of any size $m < n$.
We further note that this hypothesis 
will always be fulfilled 
if instead of considering $P$
we consider its lazy version $(P + Id)/2$.
\begin{lem}
	For any 
	$$
		1 = \theta_0 > \theta_1 \geq \theta_2 \geq \dots \geq \theta_{m - 1} \geq 0
	$$
	there always exists a {\em reversible and irreducible} stochastic matrix $\bar P$ with such a spectrum.
\end{lem}
\begin{proof}
	Let us set 
	$$
		A = \left(
			\begin{array}{cccccc}
				1		&	-1		& 	0		&	\cdots	&	\cdots	&	0		\\
				1		&	1		& 	-2		&	\ddots	&			&	\vdots	\\
				1		&	1		& 	1		&	-3		&	\ddots	&	\vdots	\\
				\vdots	&	\vdots	& 	\vdots	&	\ddots	&	\ddots	&	0		\\
				\vdots 	&	\vdots	&   \vdots 	&			&	\ddots	& -(m - 1)	\\
				1      	&	1     	&   1      	& \cdots	& 	\cdots	& 	1		
			\end{array}
		\right)
        \,,
	$$
	a matrix with orthogonal rows,
	and introduce the diagonal matrices
	$$
		D_\theta = \left(
			\begin{array}{ccccc}
				\theta_0 &&&&\\
				& \theta_1 &&&\\
				&& \ddots &&\\
				&&& \ddots &\\
				&&&& \theta_{m-1}
			\end{array}
		\right),
		\qquad
		D_{\bar \mu} = \left(
			\begin{array}{ccccc}
				\frac1{1 \times 2} &&&&\\
				& \frac1{2 \times 3} &&&\\
				&& \ddots &&\\
				&&& \frac1{(m -1)m} &\\
				&&&& \frac1{m}
			\end{array}
		\right),
	$$
	the second one being such that $Q = D^{1/2}_{\bar\mu} A$ is orthogonal.
	We compute 
	$$
		\bar P = D^{-1/2}_{\bar\mu} Q D_\theta Q^t D^{1 /2}_{\bar\mu} = A D_\theta A^t D_{\bar\mu}
	$$
	to find
	\begin{equation} \label{fluo}
		\bar P = \left(
			\begin{array}{cccccc}
				\frac{\Sigma_1 + \theta_1}{1 \times 2} & \frac{\Sigma_1 - \theta_1}{2 \times 3} & \frac{\Sigma_1 - \theta_1}{3 \times 4} &
				\dots & \frac{\Sigma_1 - \theta_1}{(m -1)m} & \frac{\Sigma_1 - \theta_1}{m} \\
				\frac{\Sigma_1 - \theta_1}{1 \times 2} & \frac{\Sigma_2 + 2^2\theta_2}{2 \times 3} & \frac{\Sigma_2 - 2\theta_2}{3 \times 4} &
				\dots & \frac{\Sigma_2 - 2\theta_2}{(m - 1)m} & \frac{\Sigma_2 - 2\theta_2}{m} \\
				\frac{\Sigma_1 - \theta_1}{1 \times 2} & \frac{\Sigma_2 - 2\theta_2}{2 \times 3} & \frac{\Sigma_3 + 3^2\theta_3}{3 \times 4} &
				\dots & \frac{\Sigma_3 - 3\theta_3}{(m - 1)m} & \frac{\Sigma_3 - 3\theta_3}{m} \\
				\vdots & \vdots & \vdots & \ddots & \vdots & \vdots\\
				\frac{\Sigma_1 - \theta_1}{1 \times 2} & \frac{\Sigma_2 - 2\theta_2}{2 \times 3} & \frac{\Sigma_3 - 3\theta_3}{3 \times 4} &
				\dots & \frac{\Sigma_{m -1} + (m - 1)^2\theta_{m - 1}}{(m - 1)m} & \frac{\Sigma_{m -1} - (m - 1)\theta_{m - 1}}{m} \\
				\frac{\Sigma_1 - \theta_1}{1 \times 2} & \frac{\Sigma_2 - 2\theta_2}{2 \times 3} & \frac{\Sigma_3 - 3\theta_3}{3 \times 4} &
				\dots & \frac{\Sigma_{m -1} - (m - 1)\theta_{m - 1}}{(m - 1)m} & \frac{\Sigma_{m}}{m} 
			\end{array}
		\right)
	\end{equation}
	with, for all $1 \leq k \leq m$, $\Sigma_k = \sum_{j < k} \theta_j$.
	$\bar P$ is stochastic, irreducible and reversible with respect to $\bar\mu$ defined by
	$$
		\bar\mu(k) = \left\{
			\begin{array}{cl}
				\frac{1}{k(k + 1)} &\mbox{if $k < m$,}\\
				\frac{1}{m} &\mbox{if $k = m$.}
			\end{array}
		\right.
	$$
It also has the desired spectrum.
\end{proof}
\noindent {\it Comment:} The proof actually shows that the positivity hypothesis on the $\theta_{j}$'s can be slightly relaxed:
we only have to require the numerators of the diagonal coefficients
in~\eqref{fluo} to be non-negative.

We conclude this section by observing 
that the universal solution we just provided
is not fully satisfactory.
First, it requires a detailed knowledge
of the spectrum that can be practically unavailable.
Second, we can expect such a universal solution
to produce very squeezed solutions.
Indeed, the coefficients $C(\bar x, j)$ in~\eqref{manuel}
will be given by the matrix $C = D^{-1 / 2}_{\bar\mu} Q = A$
or by $C = A D_{\delta}$ with $D_{\delta}$ a rescaling diagonal matrix
$$
    D_\delta = \left(
        \begin{array}{cccc}
            1&&&\\
            &\delta_1&&\\
            &&\ddots&\\
            &&&\delta_{m - 1}
        \end{array}
    \right)
$$
ensuring the non-negativity of the $\nu_{\bar x}$.
The fact that the $\delta_i$'s may have to be chosen very small
can be the source of very strong squeezing.

\section{Proof of Theorem~\ref{cornet}}\label{jazz}
\subsection{Total variation estimates}
Inequality~\eqref{TVDF.eq} is a direct consequence of Inequality~\eqref{TVDFm.eq} and Proposition \ref{nombre.prop}. Indeed,
\begin{align*}
 \E_q \cro{\sum_{\bx \in \bX} d_{TV}(\Lambda P(\bx,\cdot), \bP \Lambda(\bx, \cdot))} 
&= \sum_{i=1}^{n}  \E_q \cro{ \sum_{\bx \in \bX} d_{TV}(\Lambda P(\bx,\cdot), \bP \Lambda(\bx, \cdot)) 
		\biggm| {\abs{\bX}=i}} \P_q\cro{\abs{\bX}=i} 
\\
& \leq \sum_{i=1}^{n} \frac{q' (n-i)}{\alpha} 	\P_q\cro{\abs{\bX}=i}
\\
&= 	\frac{q'}{\alpha}  \E_q\cro{n-\abs{\bX}} \, . 
\end{align*}
It remains thus to prove \eqref{TVDFm.eq}. Applying Markov property at time $\tau_1$, we get 
\[ \bP (\bx, \cdot) = \sum_{y \in \XX} P(\bx,y) P_y\cro{X(H_{\bX})= \cdot}  \, .
\] 
Moreover, set $\delta_{\bx}$ the Dirac measure at $\bx$, seen both as a probability measure
on $\XX$ and as a row vector of dimension $n$. Then, we can rewrite
\begin{align*}
 \Lambda P (\bx, \cdot) & = \delta_{\bx} K_{q'}P(\cdot) =  \delta_{\bx} P K_{q'} (\cdot) 
		= \sum_{y \in \XX} P(\bx,y)  P_y\cro{X(T_{q'}) = \cdot}    
\\   & = \sum_{y \in \XX} P(\bx,y)  P_y\cro{H_{\bX} < T_{q'} ; X(T_{q'}) = \cdot} 
	+  \sum_{y \in \XX} P(\bx,y)  P_y\cro{H_{\bX} \geq T_{q'} ; X(T_{q'}) = \cdot} 
\\ & = \sum_{y \in \XX, \bz \in \bX} P(\bx,y)   P_y\cro{H_{\bX} < T_{q'}; X(H_{\bX})=\bz} 
	P_{\bz} \cro{X(T_{q'}) = \cdot} 
\\ & \hspace*{1cm}  +  \sum_{y \in \XX} P(\bx,y)  P_y\cro{H_{\bX} \geq T_{q'} ; X(T_{q'}) = \cdot}  
 \mbox{ (by Markov property at time } H_{\bX}) \, , 
\\ & = \bP \Lambda(\bx,\cdot) - \sum_{y \in \XX, \bz \in \bX} P(\bx,y) P_y\cro{H_{\bX} \geq T_{q'}; X(H_{\bX})=\bz} 
	P_{\bz} \cro{X(T_{q'}) = \cdot} 
\\ & \hspace*{1cm}  +  \sum_{y \in \XX} P(\bx,y)  P_y\cro{H_{\bX} \geq T_{q'} ; X(T_{q'}) = \cdot}  \, .
\end{align*}
Therefore, 
\begin{eqnarray*}
d_{TV}(\Lambda P (\bx, \cdot), \bP \Lambda(\bx,\cdot))
& = & \frac{1}{2} \sum_{x \in \XX} \abs{\Lambda P (\bx, x)-  \bP \Lambda(\bx,x)}
\\ & \leq & \frac{1}{2} \sum_{x \in \XX, y \in \XX, \bz \in \bX} P(\bx,y) P_y\cro{H_{\bX} \geq T_{q'}; X(H_{\bX})=\bz} 
	P_{\bz} \cro{X(T_{q'}) = x} 
\\ && \hspace*{1cm}  
+  \frac{1}{2} \sum_{x \in \XX, y \in \XX} P(\bx,y)  P_y\cro{H_{\bX} \geq T_{q'} ; X(T_{q'}) = x}
\\ & = &\sum_{ y \in \XX} P(\bx,y)  P_y\cro{H_{\bX} \geq T_{q'}} 
\label{TVDFas.eq}
\\ & = &\sum_{ y \in \XX} P(\bx,y) E_y\cro{1-e^{-q' H_{\bX}}}Ê
\\ & \leq& \sum_{ y \in \XX} P(\bx,y) E_y\cro{ q' H_{\bX}}  = q'E_{\bar x} \left[
    H_{\bar{\mathcal X}}^+ - \tau_1
\right]
\, . 
\end{eqnarray*}
We now take the expectation with respect to $\E_q$. 
$$
\E_q\cro{\sum_{\bx \in \bX} d_{TV}(\Lambda P (\bx, \cdot), \bP \Lambda(\bx,\cdot)) \biggm|{\abs{\bX}=m}}
 \leq q'\E_q\cro{\sum_{\bx \in \bX} E_{\bx}\cro{H^+_{\bX} - \tau_1} \Bigm| \abs{\bX} =m} \,.
$$
Formula~\eqref{luna} gives then the desired result.

\subsection{Squeezing estimates}
\label{flatDF.sec}
We now prove the quantitative upper bounds on the squeezing of $\Lambda$
stated in~\eqref{flatDF.eq}.
We begin with the following lemma:
\begin{lem} 
\label{flatDF.lem}
For any $m \in \acc{1,\cdots, n}$,
\begin{equation}
\E_q\cro{\SS(\Lambda) \Bigm|{\abs{\bX}=m}} \leq 
\frac{\sqrt{\sum_{\abs{J}=m-1} \prod_{j \in J} p'^2_j} 
\sqrt{\sum_{\abs{J}=m} \prod_{j \in J} p'^{-2}_j \, \prod_{j \in J} p_j^2 \, \prod_{j \notin J} (1-p_j)^2 }}
{\sum_{\abs{J}=m} \prod_{j \in J} p_j  \, \prod_{j \notin J} (1-p_j) }
\, .
\end{equation}
\end{lem}

\noindent

\begin{proof}
Note first that 
\[ \SS(\Lambda)^2 =\sum_{\bx \in \bX} \Gamma^{-1}(\bx,\bx) 
= \sum_{\bx \in \bX} \frac{\det_{\bX \setminus \acc{\bx}}(\Gamma)}{\det(\Gamma)}
= \sum_{\bx \in \bX} \frac{\Vol^2(\nu_{\by}; \by \in \bX, \by \ne \bx)}{\Vol^2(\nu_{\by}; \by \in \bX)}
\, .
\] 
Hence,
\begin{equation}
\label{flat.eq}
 \E_q\cro{ \SS(\Lambda) \Bigm|{\abs{\bX}=m}}
=   \sum_{\abs{R}=m} \P_q\cro{ \bX = R \Bigm|{\abs{\bX}=m}}
 \frac{\sqrt{\sum_{\bx \in R} \Vol^2(\nu_{\by}; \by \in R, \by \ne \bx)}}{\sqrt{\Vol^2(\nu_{\by}; \by \in R)}}
\, .
 \end{equation}
 From Proposition \ref{racines.prop}, $\bX=\rho(\Phi)$ is a determinantal process associated to the kernel
 $K_{q}$. Remind that for all $j \in \acc{0,\cdots, n-1}$, $\mu_j (-\LL)= \la_j \mu_j$. The $\mu_j$ are 
 orthogonal by symmetry of $-\LL$, and we assume that for all $j \in \acc{0,\cdots, n-1}$, $\nors{\mu_j}=1$, so that
 $\mu_0=\mu$. Hence, we get  $\mu_j K_{q}= \frac{q}{q+\lambda_j} \mu_j$. One way to construct 
 $\rho(\Phi)$, the number of roots being fixed equal to $m$, is  to choose  $m$ eigenvectors  of $K_{q}$, according
 to Bernoulli random variables with parameters $p_j$, and then to choose $\bX$ according to the determinantal process associated
 to the projector operator onto the $m$ chosen eigenvectors. More formally, 
 \begin{equation}
 \label{ProbCondbX.eq}  
 \P_q\cro{\bX = R \Bigm|{\abs{\bX}=m}} 
 = \frac{1}{Z_{m,q}}
 \sum_{\abs{J} = m} \prod_{j \in J} \frac{q}{q+\lambda_j}
 \, \prod_{j \notin J} \frac{\lambda_j}{q+\lambda_j} \,  
 \det^2\pare{\bras{\frac{\delta_{\bx}}{\nors{\delta_{\bx}}};\mu_j}_{ \bx \in R,  j \in J}}
 \, ,
 \end{equation}
 where  $Z_{m,q}$ is  a normalizing constant ($Z_{m,q}= \P_q \cro{\abs{\bX}=m}$). We go back to \eqref{flat.eq}
 and turn to the term $\Vol^2(\nu_{\by}; \by \in R)$. It follows from Cauchy-Binet formula that 
 \[ \Vol^2(\nu_{\by}; \by \in R) =\sum_{\abs{J}=m} \det^2\pare{\bras{\nu_{\by},\mu_j}, \by \in R, j\in J }
 \, .
 \]
 Note that $\nu_{\by} = \delta_{\by} K_{q'} = \sum_{j=0}^{n-1} \bras{ \delta_{\by}; \mu_j} \mu_j K_{q'}
 =  \sum_{j=0}^{n-1} p'_j \bras{ \delta_{\by}; \mu_j}  \mu_j$. Thus $\bras{\nu_{\by},\mu_j}= p'_j \bras{ \delta_{\by}; \mu_j}$.
 We obtain then 
  \begin{equation}
  \label{Vol.eq}
   \Vol^2(\nu_{\by}; \by \in R) =\sum_{\abs{J}=m} \prod_{j \in J} p'^2_j  \, \, 
  \det^2\pare{ \bras{ \delta_{\by}; \mu_j},\by \in R, j\in J }
  \, .
  \end{equation}
 Putting \eqref{ProbCondbX.eq}  and \eqref{Vol.eq} into \eqref{flat.eq}, we are led to 
 \begin{align*} 
 \E_q\cro{ \SS(\Lambda) \Bigm|{\abs{\bX}=m}}
= \frac{1}{Z_{m,q}} \sum_{\abs{R}=m} &
\frac{\sqrt{\som{\bx \in R}{} \Vol^2(\nu_{\by}; \by \in R, \by \ne \bx)}}
{\pro{\bx \in R}{} \nor{\delta_{\bx}}^{*2}}\\
&\quad\times
\frac{ \som{\abs{J} = m}{} \pro{j \in J}{} p_j \, \, \pro{j \notin J}{} (1-p_j) \,\, 
\det^2\pare{\bras{\delta_{\bx};\mu_j},\bx \in R,  j \in J}}
{\sqrt{\som{\abs{J}=m}{} \pro{j \in J}{} p'^2_j  \, \, 
  \det^2\pare{ \bras{ \delta_{\by}; \mu_j}, \by \in R, j\in J }}}
\, .
\end{align*}
 Cauchy-Schwartz inequality then yields
\begin{align*} 
 \E_q\cro{ \SS(\Lambda) \Bigm|{\abs{\bX}=m}}
 \leq \frac{1}{Z_{m,q}} \sum_{\abs{R}=m}  &
 \frac{\sqrt{\som{\bx \in R}{} \Vol^2(\nu_{\by}; \by \in R, \by \ne \bx)}}
{\pro{\bx \in R}{} \nor{\delta_{\bx}}^{*2}} \\
& \times \sqrt{\som{\abs{J} = m}{} \pro{j \in J}{} \frac{p_j^2}{p'^2_j} \,\, \pro{j \notin J}{} (1-p_j)^2 \,\, 
\det^2\pare{\bras{\delta_{\bx};\mu_j},\bx \in R,  j \in J}}\,,
\end{align*}
and
\begin{align*} 
 \E_q\cro{ \SS(\Lambda) \Bigm|{\abs{\bX}=m}}
 \leq \frac{1}{Z_{m,q}} &
 \sqrt{\som{\abs{R}=m}{}   \frac{\som{\bx \in R}{} \Vol^2(\nu_{\by}; \by \in R, \by \ne \bx)}
 {\pro{\bx \in R}{} \nor{\delta_{\bx}}^{*2}}
 } \\
 &\times \sqrt{\som{\begin{array}{c}\scriptstyle{ \abs{R}=m }\\[-2pt] 
 		\scriptstyle{\abs{J} = m }\end{array}}{} 
 \pro{j \in J}{} \frac{p_j^2}{p'^2_j} \,\, \pro{j \notin J}{} (1-p_j)^2 \,\, 
 \det^2\pare{\bras{\frac{\delta_{\bx}}{\nors{\delta_{\bx}}};\mu_j};  \bx \in R,  j \in J}
 }\,.
 \end{align*}
 Using again Cauchy-Binet formula, we get
 \[ \sum_{\abs{R}=m}  \det^2\pare{\bras{\frac{\delta_{\bx}}{\nors{\delta_{\bx}}};\mu_j},\bx \in R,  j \in J}
 = \Vol^2(\mu_j, j \in J) 
 =1
 \, ,
 \]
 so that the term in the second square root is equal to 
 \[ \sum_{\abs{J}=m}  \pro{j \in J}{} \frac{p_j^2}{p'^2_j} \,\, \pro{j \notin J}{} (1-p_j)^2 \, .
 \]
 We turn now to the term in the first square root, which can be rewritten, by using twice the Cauchy-Binet formula, as
 \begin{align*}
&  \som{\bx \in \XX}{}  \frac{1}{\nor{\delta_{\bx}}^{*2}} \som{\abs{R}=m, \bx \in R}{} 
  \frac{\Vol^2(\nu_{\by}; \by \in R, \by \ne \bx)}{\pro{\bx \in R \setminus \acc{\bx}}{} \nor{\delta_{\bx}}^{*2}}
 \\
& \hspace*{1cm}
 = \som{\bx \in \XX}{}  \mu(\bx)  \som{R \subset \XX \setminus \acc{\bx}, \abs{R}=m-1}{} 
\frac{\Vol^2(\nu_{\by}; \by \in R)}{\pro{\by \in R}{} \nor{\delta_{\by}}^{*2}}
\\
& \hspace*{1cm}
 = \som{\bx \in \XX}{}  \mu(\bx)  \som{R \subset \XX \setminus \acc{\bx}, \abs{R}=m-1}{} \, 
 \som{ \abs{J}=m-1}{} \,  \pro{j\in J}{} p'^2_j \,\, \det^2\pare{\bras{\frac{\delta_{\by}}{\nors{\delta_{\by}}}; \mu_j};\by \in R, j\in J}
 \\
& \hspace*{1cm}
\leq \som{\bx \in \XX}{}  \mu(\bx)  \som{ \abs{J}=m-1}{} \, \pro{j\in J}{} p'^2_j \, 
\som{ \abs{R}=m-1}{}  \det^2\pare{\bras{\frac{\delta_{\by}}{\nors{\delta_{\by}}}; \mu_j},\by \in R, j\in J}
\\
& \hspace*{1cm}
 =   \som{ \abs{J}=m-1}{} \, \pro{j\in J}{} p'^2_j \,  \Vol^2(\mu_j, j\in J)
\\
& \hspace*{1cm}
 =  \som{ \abs{J}=m-1}{} \, \pro{j\in J}{} p'^2_j
  \, . 
 \end{align*}
To end the proof of the lemma, it 
is sufficient  to note that 
\[
Z_{m,q} 
= {\mathbb P}_q\left(
    \bigl|\bar{\mathcal X}\bigr| = m
\right) = \som{\abs{J}=m}{}  \pro{j \in J}{} p_j \, \,  \pro{j \notin J}{} (1-p_j) \, .
\]
 \end{proof}

\allowdisplaybreaks
We can now conclude the proof of~\eqref{flatDF.eq} and of Theorem~\ref{cornet}.
Note that $\sum_{\abs{J} = m-1} \prod_{j \in J} p'^2_j$ is the coefficient of $t^{m-1}$ in the polynomial
$P(t)=\prod_{j=0}^{n-1} (1 + t p'^2_j)$. Hence for any $t>0$, one gets 
\[  \sum_{\abs{J} = m-1} \prod_{j \in J} p'^2_j \leq \frac{1}{t^{m-1}}  \prod_{j=0}^{n-1} (1 + t p'^2_j)
= \frac{1+t}{t^{m-1}}  \prod_{j=1}^{n-1} (1 + t p'^2_j) \, . 
\]
In the same way, for any $x>0$,  
\begin{align*}
 \sum_{\abs{J} = m-1, J \subset \acc{1,\cdots, n-1}} 
 & \prod_{j \in J} \frac{p^2_j}{p'^2_j} 
\prod_{j \in \acc{1,\cdots, n-1} \setminus J} (1-p_j)^2
\\
& = \prod_{j=1}^{n-1}(1-p_j)^2 
\sum_{\abs{J} = m-1, J \subset \acc{1,\cdots, n-1}} \prod_{j \in J} \frac{p^2_j}{p'^2_j (1-p_j)^2}
\\
& \leq \frac{ \prod_{j=1}^{n-1}(1-p_j)^2 }{x^{m-1}}  \prod_{j=1}^{n-1} \pare{1 + x \frac{p^2_j}{p'^2_j (1-p_j)^2}}
\\
& = \frac{1}{x^{m-1}}  \prod_{j=1}^{n-1} \pare{ (1-p_j)^2  + x \frac{p^2_j}{p'^2_j }}
 \, . 
\end{align*}
Hence, for any $x,t >0$, for any $m \in \acc{1,\cdots,n}$,
\[
\E_q\cro{\SS(\Lambda)\Bigm|{\abs{\bX}=m}}
\leq \frac{1}{\P_q\cro{\abs{\bX}=m}} \frac{\sqrt{1+t}}{(tx)^{\frac{m-1}{2}}}
\sqrt{ \prod_{j=1}^{n-1} (1+t p'^2_j)  \pare{ (1-p_j)^2  + x \frac{p^2_j}{p'^2_j }}}
\, .
\]
One can check that 
\[ (1+t p'^2_j)  \pare{ (1-p_j)^2  + x \frac{p^2_j}{p'^2_j }}
= \pare{1+ (\sqrt{xt}-1) p_j}^2 + \pare{\sqrt{t} p'_j (1-p_j) - \sqrt{x} \frac{p_j}{p'_j}}^2
\, .
\] 
Take now $xt=1$. We obtain that for any $t >0$, for any $m \in \acc{1,\cdots,n}$,
\begin{align*}
\E_q\cro{\SS(\Lambda)  \ind_{\abs{\bX}=m}}
& \leq \sqrt{1+t} \, \sqrt{ \prod_{j=1}^{n-1} \pare{ 1+ \pare{\sqrt{t} p'_j (1-p_j)  - \frac{1}{\sqrt{t}}  \frac{p_j}{p'_j }}^2}}
\\
& \leq \sqrt{1+t} \, \exp \pare{ \frac 1 2 \sum_{j=1}^{n-1}  \pare{\sqrt{t} p'_j (1-p_j)  - \frac{1}{\sqrt{t}}  \frac{p_j}{p'_j }}^2 }
\\
& = \sqrt{1+t} \, \exp  \pare{\frac t  2 \sum_{j=1}^{n-1} p'^2_j (1-p_j)^2 + \frac{1}{2t} \sum_{j=1}^{n-1} \frac{p_j^2}{p'^2_j }
		-  \sum_{j=1}^{n-1} p_j (1-p_j)}
\, .
\end{align*}
Optimizing the exponential term in $t$ and choosing $t = T_n$ lead to \eqref{flatDF.eq}. 

\section{Proof of Theorem~\ref{TVmeta.theo}} \label{dizzy}
 Let us first rewrite $K_{q'}$ in terms of $\mu$.  
 \begin{lem}
 For   $x \in \XX$, 
 \[ K_{q'}(x,\cdot) = \E_{q'} \cro{\mu_{A'(x) \cap A(\rho'_x)}(\cdot)} \, , \P_q \, \mbox{a.s.}.
 \] 
 \end{lem}

\begin{proof}
Starting Wilson's algorithm from $x$ to construct $\Phi'$, we get 
\begin{align*}
K_{q'}(x,y) & = P_x\cro{X(T_{q'}) = y} 
\\
& = \P_{q'}\cro{ \rho'_x=y}
\\
& = \E_{q'}\cro{\P_{q'}\cro{\rho'_x=y \bigm| {\AA(\Phi')}}}
\\
& = \E_{q'} \cro{\mu_{A'(x)}(y)} \, , 
\end{align*}
where the last equality comes from Proposition \ref{rootsgivenpartition.prop}. Hence,
$\P_q$ a.s., 
 \begin{align*}
K_{q'}(x,y)
& = \sum_{\bx \in \bX} \E_{q'} \cro{ \mu_{A'(x)\cap A(\bx)}(y) \, \mu_{A'(x)}(A(\bx))} 
\\
&= \sum_{\bx \in \bX} \E_{q'} \cro{ \mu_{A'(x)\cap A(\bx)}(y) \P_{q'}\cro{\rho'_x \in A(\bx) \bigm| {\AA(\Phi')}}}
\\
& = \sum_{\bx \in \bX} \E_{q'} \cro{ \mu_{A'(x)\cap A(\bx)}(y) \ind_{A(\bx)}(\rho'_x)}
\\
& =  \E_{q'} \cro{ \mu_{A'(x)\cap A(\rho'_x)}(y) } \, .
\end{align*}

 \end{proof}

 \begin{lem}
 For any   $x \in \XX$, set $ \tilde{K}_{q'}(x,\cdot) = \E_{q'} \cro{\mu_{A(\rho'_x)}(\cdot)}$. Then, $\P_q$ a.s.,
 \[
  \Lambda \tilde{K}_{q'} = \bP  \Lambda \,  .
  \]
 \end{lem}
\begin{proof} $\P_q$ a.s.,  for any   $x, y \in \XX$, 
 \[ \tilde{K}_{q'}(x,y) = \E_{q'} \cro{\mu_{A(\rho'_x)}(y)}
 = \sum_{\by \in \bX} \mu_{A(\by)}(y) \P_{q'} \cro{\rho'_x \in A(\by)}
 = \sum_{\by \in \bX} \nu_{\by}(y) P_x \cro{X(T_{q'}) \in A(\by)} \, .
 \]
 Hence,  $\P_q$ a.s., for any $\bx \in \bX$, and $y \in \XX$,
\[
\nu_{\bx} \tilde{K}_{q'} (y)= \sum_{x \in \XX}  \sum_{\by \in \bX} \nu_{\bx}(x)  \nu_{\by}(y)
P_x \cro{X(T_{q'}) \in A(\by)}
= \sum_{\by \in \bX}  \nu_{\by}(y) P_{\nu_{\bx}} \cro{X(T_{q'}) \in A(\by)}
= \bP \Lambda(\bx,y)
\, .
\]
\end{proof}
Therefore, $\P_q$ a.s., for any $\bx \in \bX$,
\begin{align*}
d_{TV}(\Lambda K_{q'}(\bx,\cdot), \bP  \Lambda (\bx,\cdot)  )
& = d_{TV}(\Lambda K_{q'}(\bx,\cdot), \Lambda \tilde{K}_{q'}(\bx,\cdot) ) 
\\
& \leq   \sum_{x \in \XX}  \nu_{\bx}(x)  \,  d_{TV}(K_{q'}(x,\cdot),  \tilde{K}_{q'}(x,\cdot))
\\
& \leq  \sum_{x \in \XX}  \nu_{\bx}(x) \E_{q'}\cro{d_{TV}(\mu_{A'(x)\cap A(\rho'_x)},\mu_{A(\rho'_x)})} \, .
\end{align*}
When $B$ is a subset of $C$, one has $d_{TV}(\mu_B,\mu_C) = \mu_C(B^c)$. This
yields
\[ 
d_{TV}(\Lambda K_{q'}(\bx,\cdot), \bP  \Lambda (\bx,\cdot)  ) 
\leq  \sum_{x \in \XX}  \nu_{\bx}(x)  \E_{q'}\cro{ \mu_{A(\rho'_x)}(A'(x)^c)}
=  \sum_{x \in \XX}  \nu_{\bx}(x)   \E_{q'}\cro{\P_q\cro{\rho_{\rho'_x} \notin A'(x)|\AA(\Phi)}}
\]
Note that 
\[ \sum_{\bx \in \bX}  \nu_{\bx}(x) = \sum_{\bx \in \bX}  \frac{\mu(x)}{\mu(A(\bx))} \ind_{A(\bx)}(x)
=  \sum_{\bx \in \bX}  \frac{\mu(x)}{\mu(A(x))} \ind_{A(\bx)}(x) 
= \frac{\mu(x)}{\mu(A(x))}  \sum_{\bx \in \bX}  \ind_{A(\bx)}(x) 
=  \frac{\mu(x)}{\mu(A(x))}  
\, .
\]
Summing on $\bx$ and integrating w.r.t. $\E_q$, leads to 
\[ 
\E_q\cro{\sum_{\bx \in \bX} d_{TV}(\Lambda K_{q'}(\bx,\cdot), \bP  \Lambda (\bx,\cdot)  ) }
\leq 
\sum_{x \in \XX} \E_{q,q'}\cro{\mu_{A(x)}(x) \ind_{A'(x)^c}(\rho_{\rho'_x})}
\, .
\]
Let $p\geq1$ and $p^*$ its conjugate exponent. Using H\"older's inequality, we get
\begin{align*}
\E_q\cro{\sum_{\bx \in \bX} d_{TV}(\Lambda K_{q'}(\bx,\cdot), \bP  \Lambda (\bx,\cdot)  ) }
& \leq 
\pare{\sum_{x \in \XX} \E_{q,q'}\cro{ \mu_{A(x)}(x)^p}}^{1/p} 
\pare{\sum_{x \in \XX} \P_{q,q'}\cro{\rho_{\rho'_x} \notin A'(x)}}^{1/p^*}
\\
& \leq \pare{\sum_{x \in \XX} \E_{q}\cro{ \mu_{A(x)}(x)}}^{1/p} 
\pare{\sum_{x \in \XX} \P_{q,q'}\cro{\rho_{\rho'_x} \notin A'(x)}}^{1/p^*}
\, . 
\end{align*}
Note that $\sum_{x \in \XX} \E_{q}\cro{ \mu_{A(x)}(x)} = \sum_{x \in \XX} \P_q\cro{\rho_x=x}
= \sum_{x \in \XX} \P_q\cro{x\in \rho(\Phi)} = \E_{q}\cro{ \abs{\rho(\Phi)}}$. Therefore,
\begin{equation}
\label{TVmeta.ineq}
\E_q\cro{\sum_{\bx \in \bX} d_{TV}(\Lambda K_{q'}(\bx,\cdot), \bP  \Lambda (\bx,\cdot)  ) }
\leq \pare{ \E_{q}\cro{ \abs{\rho(\Phi)}}}^{1/p} 
\pare{\sum_{x \in \XX} \P_{q,q'}\cro{\rho_{\rho'_x} \notin A'(x)}}^{1/p^*}
\, .
\end{equation}

To conclude the proof of our theorem
we evaluate $\P_{q,q'}\cro{\rho_{\rho'_x} \notin A'(x)}$
for~$x$ any given point in~${\mathcal X}$.
\begin{lem} For any $x \in \XX$,  let  $ \Gamma'_x$ be the path  going from $x$ to $\rho'_x$ in 
$\Phi'$. Then, 
\[ \P_{q,q'}\cro{\rho_{\rho'_x} \notin A'(x)} \leq \frac{q'}{q} \E_{q'}\cro{\abs{\Gamma'_x}} 
\, . 
\]
\end{lem}
\begin{proof}
To decide wether $\rho_{\rho'_x}$ is in  $A'(x)$ or not, we do the following construction:
\begin{enumerate}
\item We begin the construction of $\Phi'$ using Wilson's algorithm
starting from $x$. Thus, we let evolve the Markov process starting from $x$ until an exponential time of parameter 
$q'$, and erase the loop. The result is an oriented path $\gamma'$ (= $\Gamma'_x$) 
without loops from $x$ to a point $y$ ($=\rho'_x$).
\item We go on with the construction of $\Phi$ with Wilson's algorithm starting from $y$. We let evolve the Markov 
process starting from $y$ until an exponential time $T_q$ of parameter $q$. The Markov process stops a a point
$v$ ($=\rho_{\rho'_x}$).
\item Finally, we continue the construction of $\Phi'$ using Wilson's algorithm starting from $v$. We let evolve the Markov process starting from $v$, and we stop it after an exponential time $T_{q'}$ of parameter $q'$, or when it reaches 
the already constructed path  $\gamma'$. At this point, we are able to decide wether $\rho_{\rho'_x}$ is in  $A'(x)$ or not,
since $\rho_{\rho'_x} \in A'(x)$ if and only if $T_{q'}$ is bigger than the hitting time of $\gamma'$.
\end{enumerate}
Using this construction, 
we get that for any self-avoiding path $\gamma'$ from $x$ to $y$, 
\[ \P_{q,q'}\cro{\rho_{\rho'_x} \notin A'(x)| \Gamma'_x= \gamma'; \rho'_x=y}
= P_y \cro{T_{q'}<H_{\gamma'}\circ \theta_{T_q}} \, ,
\]
where $\theta_t$ denotes the time shift.
Recall that $\tau_1$ is the first time of the clock process
on which $X$ is build from $\hat X$,
and let $S_i$ be the successive return times to $\gamma'$:
\[  S_0=0 \, , \, \, S_1=\inf \acc{t \geq \tau_1; X(t) \in \gamma'}=H_{\gamma'}^+ \, \, , \, \,
S_{i+1}= S_i + S_1 \circ \theta_{S_i} \, .
\]
Then,
\[ 
P_y \cro{T_{q'}<H_{\gamma'}\circ \theta_{T_q}}
= \sum_{i=0}^{\infty} P_y \cro{S_i \leq T_q < S_{i+1} ; T_{q'}< H_{\gamma'}\circ \theta_{T_q}}
\, .
\]
Now, if $S_i \leq T_q < S_i + \tau_1\circ \theta_{S_i}$, $X(T_q) \in \gamma'$, and $H_{\gamma'}\circ \theta_{T_q}
= 0 < T_{q'}$. If $T_q \geq S_i + \tau_1\circ \theta_{S_i}$, and $T_q < S_{i+1}$, $X(T_q) \notin \gamma'$, and
$ H_{\gamma'}\circ \theta_{T_q}= S_{i+1}-T_q$. Therefore, 
\begin{align*}
P_y \cro{T_{q'}<H_{\gamma'}\circ \theta_{T_q}}
& = \sum_{i=0}^{\infty} P_y \cro{S_i + \tau_1\circ \theta_{S_i} \leq T_q < T_{q' }+ T_q < S_{i+1}}
\\
& = \sum_{i=0}^{\infty} \sum_{z \in \gamma'} P_y\cro{S_i \leq T_q; X(S_i)=z} 
P_{z}\cro{\tau_1 \leq T_q < T_{q' }+ T_q<  H_{\gamma'}^+} , 
\end{align*}
using Markov property at time $S_i$. Set $\tilde{G}_q(y,z,\gamma')= E_y\cro{\sum_{i=0}^{+\infty} \ind_{S_i \leq T_q; X(S_i)=z}}$. Since $z\in \gamma'$, 
$\tilde{G}_q(y,z,\gamma')$ is the mean number of visits to the point $z$ up to time $T_q$. We have obtained that 
\[
P_y \cro{T_{q'}<H_{\gamma'}\circ \theta_{T_q}}
= \sum_{z \in \gamma'} \tilde{G}_q(y,z,\gamma') P_{z}\cro{\tau_1 \leq T_q < T_q + T_{q'} < H_{\gamma'}^+} \, . 
\]
We now use Markov property at time $\tau_1$ to write
\begin{align*}
P_{z}\cro{\tau_1 \leq T_q < T_q + T_{q'} < H_{\gamma'}^+} 
& = \sum_{u \notin \gamma'} P_z\cro{\tau_1 \leq T_q, X(\tau_1)=u} P_u(T_q < T_{q'}+T_q < H_{\gamma'})
\\ & \leq  \sum_{u \notin \gamma'} \frac{\alpha}{q+\alpha} P(z,u)  P_u(T_{q'} < H_{\gamma'})
\\ & = \sum_{u \notin \gamma'} \frac{1}{q+\alpha} w(z,u) \P_{q'}\cro{\rho'_u \neq y \bigm| {\Gamma'_x=\gamma'}} \, , 
\end{align*}
using that $\alpha P(z,u)=\LL(z,u)=w(z,u)$ for $z \neq u$. 
Integrating over $\gamma'$ and $y$, we are led to
\begin{align*}
\P_{q,q'}\cro{\rho_{\rho'_x} \notin A'(x)}
& \leq \sum_{y \in \XX} \sum_{\gamma: x \leadsto y} \sum_{z \in \gamma}  \sum_{u \notin \gamma} 
\frac{\tilde{G}_q(y,z,\gamma)}{q+\alpha} w(z,u)   \P_{q'}\cro{\rho'_u \neq y; \Gamma'_x=\gamma; \rho'_x=y}
\end{align*}
where the sum over $\gamma'$ is the sum on all self-avoiding paths going
from~$x$ to~$y$. Now, introducing for any such path~$\gamma$
\[ \FF_1(y,\gamma,u) := \acc{ \phi \, s.o.f. \, ;  \, y \in \rho(\phi),  \gamma \subset \phi,  \rho_u \neq y }
\, ,
\]
this can be rewritten, with $w(\phi) = \prod_{e \in \phi} w(e)$, as 
\begin{align*}
\P_{q,q'}\cro{\rho_{\rho'_x} \notin A'(x)}
& = \sum_{y \in \XX} \sum_{\gamma: x \leadsto y} \sum_{z \in \gamma}  \sum_{u \notin \gamma}  
\sum_{\phi \in {\mathcal F}_1(y,\gamma,u)}\frac{\tilde{G}_q(y,z,\gamma)}{q+\alpha} w(z,u) \frac{(q')^{\abs{\rho(\phi)}} w(\phi)}{Z(q')} 
\,.
\end{align*}
\begin{lem}
Let $G_q(y,z)=E_y\cro{\int_0^{T_q} \ind_{X(s) = z} \, ds}$.
Then $G_q(y,z)=\tilde{G}_q(y,z,\gamma)/(q+\alpha)$
for any self-avoiding path $\gamma$ that contains $z$ and goes from~$x$ to~$y$.
\end{lem}
\begin{proof}
Let $V_i$ be the successive return times to $z$:
\[  V_0=0 \, , \, \, V_1=\inf \acc{t \geq \tau_1; X(t)  = z} \, \, , \, \,
V_{i+1}= V_i + V_1 \circ \theta_{V_i} \, .
\]
Then $\tilde{G}_q(y,z,\gamma)= \delta_y(z)+ \sum_{i=1}^{+\infty} E_y[ \ind_{V_i \leq T_q}]$. Moreover, using Markov's property at 
time $V_i$, 
\begin{align*}
G_q(y,z) & = \sum_{i=0}^{\infty} E_y\cro{\int_{V_i}^{V_{i+1}} \ind_{T_q \geq s} \ind_{X(s)=z} \, ds}
 \\
& = \sum_{i=0}^{\infty} E_y\cro{ \ind_{V_i \leq T_q} E_{X(V_i)}\cro{\int_{0}^{V_{1}}  \ind_{T_q \geq s}  \ind_{X(s)=z}\, ds
}} \\
& = E_y\cro{\int_{0}^{V_{1}}  \ind_{T_q \geq s}  \ind_{X(s)=z} \, ds}+ 
\sum_{i=1}^{\infty}E_y\cro{ \ind_{V_i \leq T_q}} E_{z}\cro{\int_{0}^{V_{1}}  \ind_{T_q \geq s}  \ind_{X(s)=z} \, ds}
\\
&=  \pare{ \delta_y(z) + \sum_{i=1}^{\infty}E_y\cro{ \ind_{V_i \leq T_q}} } 
E_{z}\cro{\int_{0}^{V_{1}}  \ind_{T_q \geq s}  \ind_{X(s)=z} \, ds}
\\
&= \tilde{G}_q(y,z,\gamma) E_{z}\cro{\int_{0}^{V_{1}}  \ind_{T_q \geq s}  \ind_{X(s)=z} \, ds} \, . 
\end{align*}
Now, $ E_{z}\cro{\int_{0}^{V_{1}}  \ind_{T_q \geq s}  \ind_{X(s)=z} \, ds} =  
E_{z}\cro{\int_{0}^{\tau_{1}}  \ind_{T_q \geq s} \, ds}= E\cro{\tau_1 \wedge T_q} = \frac{1}{q+\alpha}
$.
\end{proof}
Hence,
\[ \P_{q,q'}\cro{\rho_{\rho'_x} \notin A'(x)}
\leq \sum_{y \in \XX} \sum_{\gamma: x \leadsto y} \sum_{z \in \gamma}  \sum_{u \notin \gamma}  
\sum_{\phi \in \FF_1(y,\gamma,u)} G_q(y,z)  w(z,u) \frac{(q')^{\abs{\rho(\phi)}} w(\phi)}{Z(q')} \, .
\] 
 We fix $y$, $\gamma$ and $z$ and want to perform  the summations over 
$u$ and $\phi$. With any pair $(u, \phi)$, with $u \notin \gamma$ and $\phi \in \FF_1(y,\gamma,u)$, we associate a new
forest $\tilde{\phi}=\tilde{\phi}(u,\phi)$ in the following way:
\begin{enumerate}
\item we reverse the edges from $z$ to $y$ along $\gamma$;
\item we add the edge $(z,u)$.
\end{enumerate}
The forest $\tilde{\phi}$ is such that:
\begin{itemize}
\item $\abs{\rho(\tilde{\phi})}= \abs{\rho(\phi)}-1$;
\item $z \notin \rho(\tilde{\phi})$. 
\item the piece $\gamma_{x \leadsto z}$ of the path $\gamma$ going from $x$ to $z$ belongs to 
$ \tilde{\phi}$;
\item the path $\overleftarrow{\gamma}_{y \leadsto z}$ consisting of the reversed path  $\gamma$  from $z$ to $y$,
belongs to $ \tilde{\phi}$. 
\end{itemize}
Using reversibility,  one has 
 $  \mu(z) \prod_{e \in \gamma_{z \leadsto y}} w(e)  = \mu(y) \prod_{e \in \overleftarrow{\gamma}_{y \leadsto z}} w(e) $,
  and
 \[ w(z,u) w(\phi)= w(\tilde{\phi}) \mu(y)/\mu(z)
 \, .
 \]
 Set $\FF_2(y,z,\gamma)=\acc{ \phi \, s.o.f. \, ;  \, z \notin \rho({\phi}), \gamma_{x \leadsto z} \subset \phi, 
 \overleftarrow{\gamma}_{y \leadsto z} \subset \phi}$.
 Note that the function 
 \[ (u,\phi) \in \acc{(u,\phi), u \notin \gamma, \phi \in \FF_1(y,\gamma,u)} 
 \mapsto \tilde{\phi} \in \FF_2(y,z,\gamma)
 \]
  is one to one. Indeed, given $\tilde{\phi}$ in $\FF_2(y,z,\gamma)$, 
 $u$ is the ``ancestor'' of $z$ in $\tilde{\phi}$, and once we know $u$, $\phi$ is obtained by cutting the edge $(z,u)$, and 
 by reversing the path $\overleftarrow{\gamma}_{y \leadsto z}$. Therefore, we obtain
\begin{align*}
 \sum_{u \notin \gamma}  
\sum_{\phi \in \FF_1(y,\gamma,u)} G_q(y,z) w(z,u) \frac{(q')^{\abs{\rho(\phi)}} w(\phi)}{Z(q')}
& = \sum_{\phi \in \FF_2(y,z,\gamma)} G_q(y,z)  \frac{\mu(y)}{\mu(z)}
\frac{(q')^{\abs{\rho(\phi)}+1} w(\phi)}{Z(q')}
\\
& =   \sum_{\phi \in \FF_2(y,z,\gamma)} G_q(z,y) 
\frac{(q')^{\abs{\rho(\phi)}+1} w(\phi)}{Z(q')}
\, ,
\end{align*}
by reversibility. At this point, we are led to
\[ \P_{q,q'}\cro{\rho_{\rho'_x} \notin A'(x)}
\leq \sum_{y \in \XX} \sum_{\gamma: x \leadsto y} \sum_{z \in \gamma}  
\sum_{\phi \in \FF_2(y,z,\gamma)} G_q(z,y)  \frac{(q')^{\abs{\rho(\phi)}+1} w(\phi)}{Z(q')} \, .
\] 
We now perform the summations over $z$ and $\gamma$ and $\phi$, $y$ being fixed. Note that 
if $\phi \in \FF_2(y,z,\gamma)$ for some $z$ and $\gamma$, $x$ and $y$ are in the same tree
($t_x=t_y$ using the notations of 
Section \ref{notation.sub}), and $z$ is their first common ancestor $a(x,y)$ in that tree. Let 
us then denote 
\[ \FF_3(y,x) = Ê\acc{ \phi \, s.o.f. \, ;  \, t_x=t_y, a(x,y) \notin \rho(\phi) } \, . 
\]
Then, 
\[ \cup_{\gamma: x \leadsto y} \cup_{z \in \gamma} \FF_2(y,z,\gamma) 
\subset \FF_3(y,x) \, .
\] 
In addition, given a forest $\phi \in \FF_3(y,x)$, there is a unique $\gamma: x \leadsto y$, and $z \in \gamma$
such that 	$\phi \in  \FF_2(y,z,\gamma)$: $z$ is the first common ancestor $a(x,y)$ of $x$ and $y$, whereas $\gamma$ is
the concatenation of the path going from $x$ to $a(x,y)$ and the reversed path from  $y$  to $a(x,y)$. 
Therefore, 
\[ 
\sum_{\gamma: x \leadsto y} \sum_{z \in \gamma}  
\sum_{\phi \in \FF_2(y,z,\gamma)} G_q(z,y)  \frac{(q')^{\abs{\rho(\phi)}+1} w(\phi)}{Z(q')} 
= \sum_{\phi \in \FF_3(y,x)} G_q(a(x,y),y) \frac{(q')^{\abs{\rho(\phi)}+1} w(\phi)}{Z(q')} \, .
\]
It remains to sum over $y$. When moving $y$ in $t_x$, $a(x,y)$ moves along the path $\gamma_x$ going 
from $x$ to the root of $t_x$. Hence, 
\begin{align*}
\sum_{y \in \XX} \, \sum_{\phi \in \FF_3(y,x)}  G_q(a(x,y),y) \frac{(q')^{\abs{\rho(\phi)}+1} w(\phi)}{Z(q')} 
& = \sum_{\phi \, s.o.f.} \, \sum_{z \in \gamma_x, z \neq \rho_x} \, \sum_{y \in t_x; a(x,y)=z} G_q(z,y) 
\frac{(q')^{\abs{\rho(\phi)}+1} w(\phi)}{Z(q')} \, .
\\
& \leq \frac{q'}{q}  \sum_{\phi \, s.o.f.} \, \sum_{z \in \gamma_x, z \neq \rho_x} \pi_{q'}(\phi)
\\
& \leq \frac{q'}{q} \E_{q'}\cro{\abs{ \Gamma'_x}}
\, .
\end{align*}
\end{proof} 

\section{Proof of Theorem~\ref{cup}} \label{chopin}

Let us first rewrite our approximate solutions of Equation~\eqref{meta.eq}
with error terms.
There are signed measures $\epsilon_{\bar x, q'}$
such that, for all $\bar x$ in $\bar{\mathcal X}$,
$$
	\mu_{A(\bar x)}K_{q'} = \sum_{\bar y \in \bar{\mathcal X}} P_{\mu_{A(\bar x)}}\!\!\left(X(T_{q'}) \in A(\bar y)\right) \mu_{A(\bar y)}  + \epsilon_{\bar x, q'}.
$$
Let us now apply the ```low-pass filter'' $MW_m$ on both sides of the equations.
On the one hand, $K_{q'}$ and $MW_m$ commute.
On the other hand, our linear independence (i.e. finite squeezing) hypothesis implies that the $\epsilon_{\bar x, q'} MW_m$
are linear combinations of the $\mu_{A(\bar x)}MW_m$. 
Indeed, since the image ${\rm im}(MW_m)$ of $MW_m$ is a vector space of dimension $m$
that contains the $m$ linearly independent $\nu_{\bar x}$, the latter should span ${\rm im}(MW_m)$.
We then get, by using the notation of the proof of Proposition~\ref{flatness.prop},
$$
	\nu_{\bar x}K_{q'}
	= \sum_{\bar y \in \bar{\mathcal X}} \left(
		P_{\mu_{A(\bar x)}}\!\!\left(X(T_{q'}) \in A(\bar y)\right) + \langle \tilde \nu_{\bar y}, \epsilon_{\bar x, q'}MW_m\rangle^*
	  \right) \nu_{\bar y}.
$$
Now, when $q'$ goes to 0, $P_{\mu_{A(\bar x)}}\!\!\left(X(T_{q'}) \in A(\bar y)\right)$ converges to $\mu(A(\bar y)) > 0$,
and, by Theorem~\ref{TVmeta.theo}, $\epsilon_{\bar x, q'}$ goes to zero.
Since our $\nu_{\bar x}$ do not depend on $q'$,
this concludes the proof of the theorem.

Let us list what would be needed to give quantitative bounds on $q'$ 
to ensure that we can build in this way exact solutions of~\eqref{meta.eq}.
We would need:
\begin{enumerate}
\item upper bounds on the $\epsilon_{\bar x, q'}$;
\item upper bounds on the $\|\tilde \nu_{\bar x}\|$;
\item lower bounds on the $P_{\mu_{A(\bar x)}}\!\!\left(X(T_{q'}) \in A(\bar y)\right)$.
\end{enumerate}
The latter are out of reach in such a general framework,
the first ones are provided by Theorem~\ref{TVmeta.theo},
the second ones would be a consequence of upper bounds on the squeezing.
This is the reason why we introduce the squeezing to measure joint overlap.
We note that given Proposition~\ref{rootsgivenpartition.prop} and Equation~\eqref{flatDF.eq}
in Theorem~\ref{cornet},
we are not so far of getting such bounds. But no convexity inequality leads here to the conclusion.

\appendix
\section{Proof of proposition~\ref{poireau}}
If such random variables exist then, for all $\bar x$, $\bar y \neq \bar x$ and $y$,
\begin{align*}
	P_{\nu_{\bar x}}\left(
		T_{\bar x} = 1, \bar Y_{\bar x}  = \bar y \Bigm| \hat X(1) = y
	\right)
 	& = \frac{
		P_{\nu_{\bar x}}\left(
			T_{\bar x} = 1, \bar Y_{\bar x}  = \bar y, \hat X(1) = y
		\right)
	}{(\nu_{\bar x}P)(y)} \\
	& = \frac{
		(1 - \bar P(\bar x, \bar x))\frac{\bar P(\bar x, \bar y)}{1 - \bar P(\bar x, \bar x)} \nu_{\bar y}(y)
	}{(\nu_{\bar x}P)(y)} \\
	& = \frac{
		{\bar P(\bar x, \bar y)}\nu_{\bar y}(y)
	}{(\nu_{\bar x}P)(y)} 
	\,.
\end{align*}
By summing on $\bar y$ we get 
\begin{align*}
 	P_{\nu_{\bar x}}\left(
		T_{\bar x} = 1 \Bigm| \hat X(1) = y
	\right)
	= \frac{
		(\nu_{\bar x}P)(y) - \bar P(\bar x, \bar x) \nu_{\bar x}(y)
	}{ 
		(\nu_{\bar x} P)(y)
	}
	= 1 - \frac{\bar P(\bar x, \bar x) \nu_{\bar x}(y)}{(\nu_{\bar x} P)(y)}
	\,.
\end{align*}
We also have
\begin{align*}
	P_{\nu_{\bar x}}\left(
		\bar Y_{\bar x} = \bar y
		\Bigm| \hat X(1) = y, T_{\bar x} = 1
	\right)
	& = \frac{
		P_{\nu_{\bar x}}\left(
			\bar Y_{\bar x} = \bar y,  T_{\bar x} = 1
			\Bigm| \hat X(1) = y
		\right)
	}{
		P_{\nu_{\bar x}}\left(
			T_{\bar x} = 1
			\Bigm| \hat X(1) = y
		\right)
	}\\
	& = \frac{
		\bar P(\bar x, \bar y) \nu_{\bar y}(y)
	}{
		(\nu_{\bar x}P)(y) - \bar P(\bar x, \bar x)\nu_{\bar x}(y)
	}
	\,.
\end{align*}

We are then led to build $T_{\bar x} \geq 1$ and $\bar Y_{\bar x}$
in the following way.
\begin{enumerate}
\item At $t = 1$ we set $T_{\bar x} = 1$ with probability
    $1 - \bar P(\bar x, \bar x)\nu_{\bar x}(\hat X(1)) / (\nu_{\bar x}P)(\hat X(1))$
	by using a uniform random variable $U_1$ which is independent of $\hat X$
	(it holds $\bar P(\bar x, \bar y)\nu_{\bar x}(y) / (\nu_{\bar x}P)(y) \leq 1$
    for all $y$ in ${\mathcal X}$ by Equation~\eqref{banzai}).
\item If we just set $T_{\bar x} = 1$ we then set $\bar Y_{\bar x} = \bar y \neq \bar x$ with a probability given by the ratio
	$\bar P(\bar x, \bar y) \nu_{\bar y}(\hat X(1))/[(\nu_{\bar x}P)(\hat X(1)) - \bar P(\bar x, \bar x)\nu_{\bar x}(\hat X(1))]$
	by using a uniform random variable $U'_1$ that is independent of $U_1$ and $\hat X$.
	(Once again~\eqref{banzai} ensures that these are positive quantities summing to one.)
\item If for all $s < t$ we did not decide to set $T_{\bar x} = s$ then we set in the same way $T_{\bar x} = t$ 
	with probability $1 - \bar P(\bar x, \bar x)\nu_{\bar x}(\hat X(t)) / (\nu_{\bar x}P)(\hat X(t))$,
	in which case we set $\bar Y_{\bar x} = \bar y \neq \bar x$ with probability
	$\bar P(\bar x, \bar y) \nu_{\bar y}(\hat X(t))/[(\nu_{\bar x}P)(\hat X(t)) - \bar P(\bar x, \bar x)\nu_{\bar x}(\hat X(t))]$.
	This is naturally done by using uniform random variable that are independent
	of $\hat X$ and $U_1$, $U_1'$, $U_2$, $U_2'$, \dots, $U_{t - 1}$,~$U_{t - 1}'$.
\end{enumerate}
At this point,
the key property to check is the stationarity of $\nu_{\bar x}$ up to $T_{\bar x}$.
To this end it suffices to check Equation~\eqref{mucca} with $t = 1$.
And one has
\begin{align*}
	P_{\nu_{\bar x}}\left(\hat X(1) = y \Bigm| T_{\bar x} > 1\right)
	& = \frac{P_{\nu_{\bar x}}\left(\hat X(1) = y,  T_{\bar x} > 1\right)}{P_{\nu_{\bar x}}\left(T_{\bar x} > 1\right)}\\
	& = \frac{P_{\nu_{\bar x}}\left(\hat X(1) = y\right) - P_{\nu_{\bar x}}\left(\hat X(1) = y,  T_{\bar x} = 1\right)}
	{1 - P_{\nu_{\bar x}}\left(T_{\bar x} = 1\right)}\\
	& = \frac{{\nu_{\bar x}}P(y) - {\nu_{\bar x}}P(y)\left(1 - \frac{\bar P(\bar x, \bar x)\nu_{\bar x}(y)}{\nu_{\bar x}P(y)}\right)}
	{1 - \sum_z {\nu_{\bar x}}P(z)\left(1 - \frac{\bar P(\bar x, \bar x)\nu_{\bar x}(z)}{\nu_{\bar x}P(z)}\right)}\\
	& = \frac{\bar P(\bar x, \bar x)\nu_{\bar x}(y)}{1 - \sum_z {\nu_{\bar x}}P(z) + \sum_z \bar P(\bar x, \bar x)\nu_{\bar x}(z)}
	= \nu_{\bar x}(y).
\end{align*}
Points (1)--(5) immediately follow.

\end{document}